\documentclass[preprint]{elsarticle}

\usepackage{amssymb,amsmath}
\usepackage{amsthm}
\usepackage{color}

\newcommand{\redcolor}[1]{#1}

\journal{Linear Algebra and its Applications}

\newtheorem{theorem}{Theorem}[section] 
\newtheorem{algorithm}{Algorithm}[section] 
\newtheorem{proposition}[theorem]{Proposition}
\newtheorem{lemma}[theorem]{Lemma}
\newtheorem{corollary}[theorem]{Corollary}

\graphicspath {{figures/}}
\begin{document}

\begin{frontmatter}

%% Title, authors and addresses

%% use the tnoteref command within \title for footnotes;
%% use the tnotetext command for theassociated footnote;
%% use the fnref command within \author or \address for footnotes;
%% use the fntext command for theassociated footnote;
%% use the corref command within \author for corresponding author footnotes;
%% use the cortext command for theassociated footnote;
%% use the ead command for the email address,
%% and the form \ead[url] for the home page:
%% \title{Title\tnoteref{label1}}
%% \tnotetext[label1]{}
%% \author{Name\corref{cor1}\fnref{label2}}
%% \ead{email address}
%% \ead[url]{home page}
%% \fntext[label2]{}
%% \cortext[cor1]{}
%% \address{Address\fnref{label3}}
%% \fntext[label3]{}

%% use optional labels to link authors explicitly to addresses:
%% \author[label1,label2]{}
%% \address[label1]{}
%% \address[label2]{}

\title{The Complexity of Primal-Dual Fixed Point Methods for Ridge Regression\tnoteref{ftpaper}}
\tnotetext[ftpaper]{The results of this paper were obtained between October 2014 and March 2015, 
during AR's affiliation with the University of Edinburgh. First version of this paper was available on August 31, 2017. This revision: January 2018.}

\author[UFPR]{{Ademir Alves Ribeiro}\corref{cor1}\fnref{cnpq}}
\ead{ademir.ribeiro@ufpr.br}

\author[UoE]{{Peter Richt\'arik}\fnref{epsrc}}
\ead{peter.richtarik@ed.ac.uk}

\fntext[cnpq]{Supported by CNPq, Brazil, Grants 201085/2014-3 and 309437/2016-4.}
\fntext[epsrc]{Supported by the EPSRC Grant EP/K02325X/1, ``Accelerated 
Coordinate Descent Methods for Big Data Optimization''.}

\cortext[cor1]{Corresponding Author}
\address[UFPR]{Department of Mathematics, Federal University of Paran\'a, CP 19081, 
81531-980, Curitiba, PR, Brazil}
\address[UoE]{School of Mathematics, University of Edinburgh, United Kingdom}

\begin{abstract}
	
\end{abstract}

\begin{keyword}
%% keywords here, in the form: keyword \sep keyword
Unconstrained minimization \sep primal-dual methods \sep ridge regression \sep fixed-point methods.

%% PACS codes here, in the form: \PACS code \sep code

%% MSC codes here, in the form: \MSC code \sep code
%% or \MSC[2008] code \sep code (2000 is the default)
\MSC 65K05 \sep 49M37 \sep 90C30
\end{keyword}

\end{frontmatter}

%% \linenumbers

%% main text

\section{Introduction}
\label{sec_intro}
Given matrices $A_1,\dots, A_n \in \mathbb{R}^{d\times m}$ encoding $n$ observations (examples), and vectors 
$y_1,\dots,y_n \in \mathbb{R}^m$ encoding associated responses (labels), one is often interested in finding a 
vector $w\in \mathbb{R}^d$ such that, in some precise sense, the product $A_i^T w$ is a good approximation of 
$y_i$  for all $i$. A fundamental approach  to this problem, used in all areas of computational practice, 
is to formulate the problem as an $L_2$-regularized least-squares problem, also known 
as {\em ridge regression}. 
In particular, we consider the {\em primal} ridge regression problem
\begin{equation}
\label{primal_ridge}
\min_{w\in\mathbb{R}^d} P(w) \stackrel{\rm def}{=} 
\frac{1}{2n} \sum_{i=1}^n \|A_i^T w - y_i\|^2 + \frac{\lambda}{2}\|w\|^2 =  
\frac{1}{2n}\|A^Tw-y\|^2+\frac{\lambda}{2}\|w\|^2,
\end{equation}
where $\lambda>0$ is a regularization parameter, $\|\cdot\|$ denotes the standard Euclidean norm. 
In the second and more concise expression we have concatenated the observation matrices and response 
vectors to form a single observation matrix 
$A=[A_1, A_2,  \cdots, A_n]\in\mathbb{R}^{d\times N}$ and a single response vector 
$y=(y_1, y_2, \cdots, y_n) \in\mathbb{R}^{N}$, where  $N=nm$.

With each observation $(A_i,y_i)$ we now associate a dual variable, $\alpha_i\in \mathbb{R}^m$. The 
Fenchel {\em dual} of \eqref{primal_ridge} is also a ridge regression problem:
\begin{equation}
\label{dual_ridge}
\max_{\alpha  \in\mathbb{R}^{N}} D(\alpha)\stackrel{\rm def}{=}
-\frac{1}{2\lambda n^2}\|A\alpha\|^2+\frac{1}{n}\alpha^T y -\frac{1}{2n}\|\alpha\|^2,\end{equation}
where $\alpha = (\alpha_1,\alpha_2,\dots,\alpha_n)\in \mathbb{R}^{N}$.

\medskip
{\noindent\bf Optimality conditions.} The starting point of this work is the observation that 
the {\em optimality conditions} for the primal and dual ridge  regression  problems can be written 
in several {\em different} ways, in the form of a linear system involving the primal and dual variables. 
In particular, we find several different matrix-vector pairs $(M,b)$, where $M\in\mathbb{R}^{(d+N)\times(d+N)}$ 
and $b \in \mathbb{R}^{d+N}$, such that the optimality conditions can be expressed in the form of a 
linear system as
\begin{equation}
\label{eq:fixed_point_system} 
x = Mx + b,
\end{equation}
where $x = (w,\alpha)\in \mathbb{R}^{d+N}$. 

\medskip
{\noindent\bf Fixed point methods.} With each system \eqref{eq:fixed_point_system} one can naturally 
associate a fixed point method performing the iteration $x^{k+1} = M x^k + b$.  However, unless the 
spectrum of $M$ is contained in the unit circle, such a method will not converge \cite{Saad}. To 
overcome this drawback, we utilize the idea of {\em relaxation}. In particular, we pick a relaxation 
parameter $\theta\neq 0$ and replace \eqref{eq:fixed_point_system} with the equivalent system
$$
x=G_{\theta} x + b_{\theta},
$$ 
where $G_{\theta} = (1-\theta) I + \theta M$ and $b_{\theta} = \theta b$.  The choice $\theta=1$ 
recovers \eqref{eq:fixed_point_system}. We then study the convergence of 
the {\em primal-dual fixed point methods} \[x^{k+1} = G_{\theta} x^k + b_{\theta}\] through a careful 
study of the spectra of the iteration matrices $G_{\theta}$.

Our work starts with the following observation:
While all these formulations are necessarily algebraically equivalent, they give rise to different  fixed-point  algorithms, with different  convergence properties. 

\subsection{Contributions and literature review}
It is well known that the role of duality in optimization and machine learning is very important, not only from the 
theoretical point of view but also computationally \cite{Shalev-Shwartz-Zhang13,Shalev-Shwartz-Zhang,Zhang}. 

However, a more recent idea that has generated many contributions is the usage of the primal and dual problems together. 
Primal-dual methods have been employed in convex optimization problems where strong duality holds, obtaining success 
when applied to several types of nonlinear and nonsmooth functions that arise in various application fields, such 
as image processing, machine learning, inverse problems, among others \cite{Chambolle-Pock,Komodakis,Quartz}.

On the other hand, fixed-point-type algorithms are classical tools for solving some structured linear systems. 
In particular, we have the iterative schemes developed by the mathematical economists Arrow, Hurwicz and Uzawa for 
solving saddle point problems \cite{Arrow-Hurwicz,Uzawa}.

In this paper we develop several primal-dual fixed point methods for the Ridge Regression problem. 
Ridge regression was introduced by Hoerl and Kennard \cite{Hoerl62,Hoerl} as a regularization method 
for solving least squares problems with highly correlated predictors. The goal is to reduce the standard 
errors of regression coefficients by imposing a penalty, in the $L_2$ norm, on their size. 

Since then, numerous papers were devoted to the study of ridge regression or even for solving problems with 
a general formulation in which ridge regression is a particular case. Some of these works have considered its 
dual formulation, proposing deterministic and stochastic algorithms that can be applied to the dual problem 
\cite{Dereny,Hawkins,Saunders,Shalev-Shwartz-Zhang13,Vinod,Zhang}.

To the best of our knowledge, the only work that considers a primal-dual fixed point approach to deal 
with ridge regression is \cite{Silva-Ribeiro-Pericaro}, where the authors deal with ill-conditioned problems. 
They present an algorithm based on the gradient method and an accelerated version of this algorithm. 

Here we propose methods based on the optimality conditions for the problem of minimizing the duality gap 
between the ridge regression problems \eqref{primal_ridge} and \eqref{dual_ridge} in different and equivalent 
ways by means of linear systems involving structured matrices. We also study the complexity of the proposed 
methods and prove that our main method achieves the optimal accelerated Nesterov rate.
\redcolor{This theoretical property is supported by numerical experiments indicating that our main 
algorithm is competitive with the conjugate gradient method.}

\subsection{Outline}
In Section \ref{sec_optim} we formulate the optimality conditions for the problem of minimizing the duality gap 
between \eqref{primal_ridge} and \eqref{dual_ridge} in two different, but equivalent, ways by means of linear 
systems involving structured matrices. 
We also establish the duality relationship between the problems \eqref{primal_ridge} and 
\eqref{dual_ridge}. In Section \ref{sec_mfp} we describe a family of (parameterized) fixed point 
methods applied to the reformulations for the optimality conditions. We present the convergence 
analysis and complexity results for these methods. Section \ref{sec_qtz} brings the main 
contribution of this work, with an accelerated version of the methods  described in 
Section \ref{sec_mfp}. In Section \ref{sec_qtze} we discuss some variants of our accelerated algorithm. 
In Section \ref{sec_num} we perform some numerical experiments. Finally, concluding remarks close our text in 
Section \ref{sec_concl}.

\section{Separable and Coupled Optimality Conditions}
\label{sec_optim}
Defining $x = (w,\alpha)\in \mathbb{R}^{d+N}$, our primal-dual problem consists 
of minimizing the duality gap between the problems \eqref{primal_ridge} and \eqref{dual_ridge}, that is  
\begin{equation}
\label{pd_prob}
\min_{x\in\mathbb{R}^{d+N}} f(x)\stackrel{\rm def}{=}P(w)-D(\alpha).
\end{equation}
This is a quadratic strongly convex problem and therefore admits a unique global solution $x^*\in\mathbb{R}^{d+N}$.

\subsection{A separable system}
Note that $\nabla f(x)=\left(\begin{array}{r}\nabla P(w) \\ -\nabla D(\alpha)\end{array}\right)$, where 
\begin{equation}
\label{gradPD}
\nabla P(w)=\frac{1}{n}A(A^Tw-y)+\lambda w
\quad\mbox{and}\quad
\nabla D(\alpha)=-\frac{1}{\lambda n^2}A^TA\alpha-\frac{1}{n}\alpha+\frac{1}{n}y.
\end{equation}
So, the first and natural way of writing the optimality conditions for problem \eqref{pd_prob} is 
just to set the expressions given in \eqref{gradPD} equal to zero, which can be written as 
\begin{equation}
\label{optM1}
\boxed{
\left(\begin{array}{r} w \\ \alpha \end{array}\right)=-\frac{1}{\lambda n}
\left(\begin{array}{cc} AA^T & 0 \\ 0 & A^TA\end{array}\right)
\left(\begin{array}{r} w \\ \alpha \end{array}\right)+\frac{1}{\lambda n}
\left(\begin{array}{c} Ay \\ \lambda ny \end{array}\right).
}
\end{equation}

\subsection{A coupled system}
In order to derive the duality between \eqref{primal_ridge} and \eqref{dual_ridge}, as well as 
to reformulate the optimality conditions for problem \eqref{pd_prob}, note that 
\begin{equation}
\label{primal_fc}
P(w)=\frac{1}{n}\sum_{i=1}^{n}\phi_i(A_i^Tw)+\lambda g(w),
\end{equation}
where $\phi_i(z)=\frac{1}{2}\|z-y_i\|^2$ and $g(w)=\frac{1}{2}\|w\|^2$. 

Now, recall that the Fenchel conjugate of a convex function $\xi:\mathbb{R}^l\to\mathbb{R}$ is 
$\xi^*:\mathbb{R}^l\to\mathbb{R}\cup\{\infty\}$ defined by 
$$
\xi^*(u)\stackrel{\rm def}{=}\sup_{s\in\mathbb{R}^l}\{s^Tu-\xi(s)\}.
$$
Note that if $\xi$ is strongly convex, then $\xi^*(u)<\infty$ for all $u\in\mathbb{R}^l$. 
Indeed, in this case $\xi$ is bounded below by a strongly convex quadratic function, 
implying that the ``$\sup$'' above is in fact a ``$\max$''. 

It is easily seen that 
$\phi_i^*(s)=\frac{1}{2}\|s\|^2+s^Ty_i$ and $ g^*(u)=\frac{1}{2}\|u\|^2$. 
Furthermore, we have 
\begin{equation}
\label{dual_fc}
D(\alpha)=-\lambda g^*\left(\frac{1}{\lambda n}\sum_{i=1}^{n}A_i\alpha_i\right)
-\frac{1}{n}\sum_{i=1}^{n}\phi_i^*(-\alpha_i). 
\end{equation}
If we write 
\begin{equation}
\label{alphabar}
\bar\alpha \stackrel{\rm def}{=} \frac{1}{\lambda n} A \alpha = \frac{1}{\lambda n}\sum_{i=1}^{n}A_i\alpha_i,
\end{equation}
the duality gap can be written as 
$$
P(w)-D(\alpha)=\lambda\big(g(w)+g^*(\bar\alpha)-w^T\bar\alpha\big)+
\frac{1}{n}\sum_{i=1}^{n}\Big(\phi_i(A_i^Tw)+\phi_i^*(-\alpha_i)+\alpha_i^TA_i^Tw\Big)
$$
and the weak duality follows immediately from the fact that 
$$
g(w)+g^*(\bar\alpha)-w^T\bar\alpha\geq 0
\quad\mbox{and}\quad
\phi_i(A_i^Tw)+\phi_i^*(-\alpha_i)+\alpha_i^TA_i^Tw\geq 0.
$$
Strong duality occurs when these quantities vanish, which is precisely the same as 
$
w=\nabla g^*(\bar\alpha)$
and $
\alpha_i=-\nabla\phi_i(A_i^Tw)
$,
or, equivalently, 
$
\bar{\alpha} = \nabla g(w)$ and $A_i^T w = \nabla \phi_i^*(-\alpha_i).
$ Therefore, another way to see the optimality conditions 
for problem \eqref{pd_prob} is by the relations 
\begin{equation}
\label{optM2w}
w=\bar\alpha=\frac{1}{\lambda n}A\alpha\quad\mbox{and}\quad \alpha=y-A^Tw.
\end{equation}
This is equivalent to 
\begin{equation}
\label{optM2}
\boxed{
\left(\begin{array}{r} w \\ \alpha \end{array}\right)=-\frac{1}{\lambda n}
\left(\begin{array}{cr} 0 & -A \\ \lambda nA^T & 0 \end{array}\right)
\left(\begin{array}{r} w \\ \alpha \end{array}\right)+
\left(\begin{array}{c} 0 \\ y \end{array}\right).
}
\end{equation}

\subsection{Compact form}
Both reformulations of the optimality conditions, \eqref{optM1} and \eqref{optM2}, can be viewed 
in the compact form 
\begin{equation}
\label{optimq}
x=Mx+b,
\end{equation}
for some $M\in\mathbb{R}^{(d+N)\times(d+N)}$ and $b\in\mathbb{R}^{d+N}$. Let us denote 
\begin{equation}
\label{M1_M2}
M_1=-\frac{1}{\lambda n}\left(\begin{array}{cc} AA^T & 0 \\ 0 & A^TA\end{array}\right)
\quad\mbox{and}\quad
M_2=-\frac{1}{\lambda n}\left(\begin{array}{cr} 0 & -A \\ \lambda nA^T & 0 \end{array}\right)
\end{equation}
the matrices associated with the optimality conditions formulated as \eqref{optM1} and 
\eqref{optM2}, respectively.
Also, let
\begin{equation}
\label{b1_b2}
b_1 = \frac{1}{\lambda n}\left(\begin{array}{c} Ay \\ \lambda ny \end{array}\right)
\quad\mbox{and}\quad
b_2=\left(\begin{array}{c} 0 \\ y   \end{array}\right).
\end{equation}
Thus, we can \redcolor{rewrite} \eqref{optM1} and \eqref{optM2} as 
\begin{equation}
\label{mfp1_mfp2}
x=M_1x+b_1\quad\mbox{and}\quad x=M_2x+b_2,
\end{equation}
respectively.

\section{Primal-Dual Fixed Point Methods}
\label{sec_mfp}
A method that arises immediately from the relation \eqref{optimq} is given by the scheme 
$$
x^{k+1}=Mx^k+b.
$$
However, unless the spectrum of $M$ is contained in the unit circle, this scheme will not 
converge. To overcome this drawback, we utilize the idea of {\em relaxation}. More precisely, 
we consider a relaxation parameter $\theta\neq 0$ and replace \eqref{optimq} with the equivalent system
$$
x=(1-\theta)x+\theta(Mx+b).
$$ 
Note that the choice $\theta=1$ recovers \eqref{optimq}.

The proposed algorithm is then given by the following framework.

\begin{center}
\boxed{
\begin{minipage}{10.5cm}
\begin{algorithm} {\bf Primal-Dual Fixed Point Method}
\label{alg_mfp}
\end{algorithm}
\begin{tabbing}
\hspace*{7mm}\= \hspace*{7mm}\= \hspace*{7mm}\= \hspace*{7mm}\=
\hspace*{7mm}\= \hspace*{7mm}\= \kill
{\sc input}: matrix $M\in\mathbb{R}^{(d+N)\times(d+N)}$, vector $b\in\mathbb{R}^{d+N}$,
parameter $\theta>0$ \\
{\sc starting point}: $x^0\in\mathbb{R}^{d+N}$ \\ 
{\sc repeat for $k=0,1,2,\dots$}
\+ \\
{\sc set} $x^{k+1}=(1-\theta)x^k+\theta(Mx^k+b)$ 
\end{tabbing}
\end{minipage}
}
\end{center}
As we shall see later, the use of the relaxation parameter $\theta$ enables us to prove 
convergence of Algorithm \ref{alg_mfp} with $M=M_1$ and $b=b_1$ or $M=M_2$ and $b=b_2$, chosen according 
to \eqref{M1_M2} and \eqref{b1_b2}, independent of the spectral radius of these matrices. 

Let us denote 
\begin{equation}
\label{G_theta}
G(\theta)=(1-\theta)I+\theta M
\end{equation}
and let $x^*$ be the solution of the problem \eqref{pd_prob}. 
Then $x^*=Mx^*+b$ with $M=M_1$ and $b=b_1$ or $M=M_2$ and $b=b_2$. Therefore, 
$
x^*=G(\theta)x^*+\theta b.
$
Further, the iteration of Algorithm \ref{alg_mfp} can be written as 
$x^{k+1}=G(\theta)x^k+\theta b$. Thus, 
\begin{equation}
\label{distxk}
\|x^k-x^*\|\leq\|G(\theta)^k\|\|x^0-x^*\|
\end{equation}
and consequently the convergence of the algorithm depends on the spectrum of $G(\theta)$. 
More precisely, it converges if the spectral radius of $G(\theta)$ is less than $1$, because 
in this case we have $G(\theta)^k\to 0$. 

In fact, we will address the following questions: 
\begin{itemize}
\item What is the range for $\theta$ so that this scheme converges?
\item What is the best choice of $\theta$?
\item What is the rate of convergence?
\item How the complexity of this algorithm compares with the known ones?
\end{itemize}

\subsection{Convergence analysis}
\label{sec_conv}
In this section we study the convergence of Algorithm \ref{alg_mfp} and answer the questions 
raised above. To this end we point out some properties of the iteration matrices and uncover 
interesting connections between the complexity bounds of the variants of the fixed 
point scheme we consider. These connections follow from a close link between the spectral 
properties of the associated matrices. 

For this purpose, let 
\begin{equation}
\label{svdA}
A=U\Sigma V^T
\end{equation}
be the singular value decomposition of $A$. That is, $U\in\mathbb{R}^{d\times d}$ and 
$V\in\mathbb{R}^{N\times N}$ are orthogonal matrices and 
\begin{eqnarray}
\label{Sigma_A}
& \Sigma = \left(\begin{array}{ccc}\widetilde\Sigma & & 0 \\ 0 & & 0 \end{array}\right) 
& \hspace{-.35cm}\begin{array}{c} p \\ d-p \end{array}  \\
& \begin{array}{ccc}
\hspace{.75cm} &  p  & N-p   \nonumber \\
\end{array}  
\end{eqnarray}
where $\widetilde\Sigma={\rm diag}(\sigma_1,\ldots,\sigma_p)$ brings the (nonzero) singular 
values $\sigma_1\geq\cdots\geq\sigma_p>0$ of $A$.

First, we state a basic linear algebra result (the proof is straightforward by induction).

\begin{proposition}
\label{pr_det}
Let $Q_j\in\mathbb{R}^{l\times l}$, $j=1,\ldots,4$, be diagonal matrices whose diagonal 
entries are components of $\alpha,\beta,\gamma,\delta\in\mathbb{R}^l$, respectively. Then 
$$
\det\left(\begin{array}{cc} Q_1 & Q_2 \\ Q_3  & Q_4 \end{array}\right)=
\prod_{j=1}^{l}\left(\alpha_j\delta_j-\beta_j\gamma_j\right). 
$$
\end{proposition}

The next result is crucial for the convergence analysis and complexity study of Algorithm \ref{alg_mfp}. 

\begin{lemma}
\label{pr_spectr}
The characteristic polynomials of the matrices $M_1$ and $M_2$, defined in \eqref{M1_M2}, 
are 
$$
p_1(t)=t^{N+d-2p}\prod_{j=1}^{p}\left(t+\frac{1}{\lambda n}\sigma_j^2\right)^2
\quad\mbox{and}\quad
p_2(t)=t^{N+d-2p}\prod_{j=1}^{p}\left(t^2+\frac{1}{\lambda n}\sigma_j^2\right),
$$
respectively.
\end{lemma}
\begin{proof}
Let $c=-\dfrac{1}{\lambda n}$. From \eqref{svdA} and \eqref{Sigma_A}, we can write 
$M_1=W\Sigma_1W^T$ and $M_2=W\Sigma_2W^T$, where $W=\left(\begin{array}{cc} U & 0 \\ 0 & V \end{array}\right)$, 
$\Sigma_1=\left(\begin{array}{cc}c\Sigma\Sigma^T & 0 \\ 0 & c\Sigma^T\Sigma \end{array}\right)$ and
\begin{eqnarray}
& \Sigma_2= \left(\begin{array}{rccrcc} 
0 & & 0 &   -c\widetilde\Sigma & & 0 \\ 
0 & & 0 &  0  & & 0  \\ 
-\widetilde\Sigma & & 0 &  0 & & 0 \\ 
0 & & 0 &  0 & & 0 
\end{array}\right)
& \hspace{-.35cm}\begin{array}{c} p \\ d-p \\ p \\ N-p. \end{array}  \nonumber \\
& \begin{array}{ccccc}
\hspace{1.25cm} &  p & d-p & \hspace{.25cm} p  & N-p   \nonumber \\
\end{array}
\end{eqnarray}
The evaluation of $p_1(t)=\det(tI-M_1)=\det(tI-\Sigma_1)$ is straightforward and 
$$
p_2(t)=\det(tI-M_2)=\det\left(\begin{array}{cccc} 
tI & 0 &   c\widetilde\Sigma & 0 \\ 0 & tI &  0  & 0  \\ 
\widetilde\Sigma & 0 &  tI & 0 \\ 0 & 0 &  0 & tI 
\end{array}\right) =\det\left(\begin{array}{cccc} 
tI & c\widetilde\Sigma & 0 & 0 \\ \widetilde\Sigma & tI &  0  & 0  \\ 
0 & 0 &  tI & 0 \\ 0 & 0 &  0 & tI 
\end{array}\right). 
$$
The result then follows from Proposition \ref{pr_det}.
\end{proof}

\medskip

The following result follows 
directly from Lemma~\ref{pr_spectr} and the fact that $M_1$ is symmetric.

\begin{corollary}
\label{spectr}
The spectral radii of $M_1$ and $M_2$ are, respectively,  
$$
\rho_1=\|M_1\|=\dfrac{\sigma_1^2}{\lambda n}=\dfrac{\|A\|^2}{\lambda n}
\quad\mbox{and}\quad
\rho_2=\dfrac{\sigma_1}{\sqrt{\lambda n}}=\dfrac{\|A\|}{\sqrt{\lambda n}}.
$$
\end{corollary}

From Corollary \ref{spectr} we conclude that if $\sigma_1<\sqrt{\lambda n}$, then 
$\rho_1\leq\rho_2<1$. So, $M_1^k\to 0$ and $M_2^k\to 0$, which in turn implies that the pure 
fixed point method, that is, Algorithm \ref{alg_mfp} with $\theta=1$, converges. However, 
if $\sigma_1\geq\sqrt{\lambda n}$, we cannot guarantee convergence of the pure method. 

Now we shall see that Algorithm \ref{alg_mfp} converges for a broad range of the 
parameter $\theta$, without any assumption on $\sigma_1$,  $\lambda$ or $n$. 
We begin with the analysis of the framework that uses $M_1$ and $b_1$, defined in 
\eqref{M1_M2} and \eqref{b1_b2}.

\subsection{Fixed Point Method based on $M_1$}

\begin{center}
\boxed{
\begin{minipage}{10.3cm}
\begin{algorithm}  {\bf Primal-Dual Fixed Point Method; $M=M_1$}
\label{alg_mfp1}
\end{algorithm}
\begin{tabbing}
\hspace*{7mm}\= \hspace*{7mm}\= \hspace*{7mm}\= \hspace*{7mm}\=
\hspace*{7mm}\= \hspace*{7mm}\= \kill
{\sc input}: $M=M_1$, $b=b_1$,
parameter $\theta>0$ \\
{\sc starting point}: $x^0\in\mathbb{R}^{d+N}$ \\ 
{\sc repeat for $k=0,1,2,\dots$}
\+ \\
{\sc set} $x^{k+1}=(1-\theta)x^k+\theta(Mx^k+b)$ 
\end{tabbing}
\end{minipage}
}
\end{center}

\begin{theorem}
\label{th_mfp1}
Let $x^0\in\mathbb{R}^{d+N}$ be an arbitrary starting point and consider the sequence 
$(x^k)_{k\in\mathbb{N}}$ generated by Algorithm~\ref{alg_mfp1} with 
$\theta\in\left(0,\dfrac{2\lambda n}{\lambda n+\sigma_1^2}\right)$. 
Then the sequence $(x^k)$ converges to the (unique) solution of the problem \eqref{pd_prob} 
at a linear rate of 
$
\rho_1(\theta)\stackrel{\rm def}{=}\max\left\{
\left|1-\theta\left(1+\dfrac{\sigma_1^2}{\lambda n}\right)\right|,1-\theta
\right\}
$ 
Furthermore, if we choose 
$\theta_1^*\stackrel{\rm def}{=}\dfrac{2\lambda n}{2\lambda n+\sigma_1^2}$, then the 
(theoretical) convergence rate is optimal and it is equal to 
$
\rho_1^*\stackrel{\rm def}{=}\dfrac{\sigma_1^2}{2\lambda n+\sigma_1^2}=
1-\theta_1^*.
$
\end{theorem}
\begin{proof}
We claim that the spectral radius of 
$G_1(\theta)\stackrel{\rm def}{=}(1-\theta)I+\theta M_1$ is $\rho_1(\theta)$ and also 
coincides with $\|G_1(\theta)\|$. Using Lemma \ref{pr_spectr}, we 
conclude that the eigenvalues of this matrix are  
$$
\left\{1-\theta-\dfrac{\theta\sigma_j^2}{\lambda n}\;,\quad j=1,\ldots,p
\right\}\cup\{1-\theta\}.
$$
So, its spectral radius is 
$$
\max\left\{
\left|1-\theta\left(1+\dfrac{\sigma_1^2}{\lambda n}\right)\right|,1-\theta
\right\}=\rho_1(\theta).
$$ 
Since $G_1(\theta)$ is symmetric, this quantity coincides with $\|G_1(\theta)\|$.
Furthermore, the admissible values for $\theta$, that is, the ones such that the eigenvalues 
have modulus less than one, can be found by solving 
$$
\left|1-\theta\left(1+\dfrac{\sigma_1^2}{\lambda n}\right)\right|<1,
$$
which immediately gives $0<\theta<\dfrac{2\lambda n}{\lambda n+\sigma_1^2}$. 
So, the linear convergence of Algorithm \ref{alg_mfp} is guaranteed for any 
$\theta\in\left(0,\dfrac{2\lambda n}{\lambda n+\sigma_1^2}\right)$.
Finally, note that the solution of the problem
$$
\min_{\theta>0}\rho_1(\theta)
$$
is achieved when 
$
\theta\left(1+\dfrac{\sigma_1^2}{\lambda n}\right)-1=1-\theta, 
$
yielding $\theta_1^*=\dfrac{2\lambda n}{2\lambda n+\sigma_1^2}$ and 
the optimal convergence rate $\rho_1^*=\dfrac{\sigma_1^2}{2\lambda n+\sigma_1^2}$. 
\end{proof}

\medskip

The top picture of Figure \ref{fig_thM1} illustrates the eigenvalues of 
$G_1(\theta)$ (magenta squares) together with the eigenvalues of 
$M_1$ (blue triangles), for a fixed value of the parameter $\theta$. The one farthest from the origin is 
$1-\theta-\dfrac{\theta\sigma_1^2}{\lambda n}$ or $1-\theta$. On the bottom we show the 
two largest (in absolute value) eigenvalues of $G_1(\theta)$ corresponding to the optimal 
choice of $\theta$. 
\begin{figure}[htbp]
\centering
\includegraphics[scale=0.45]{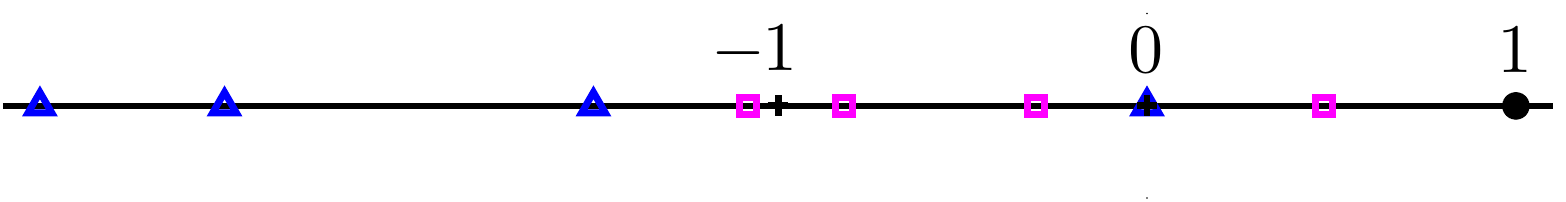}
\includegraphics[scale=0.45]{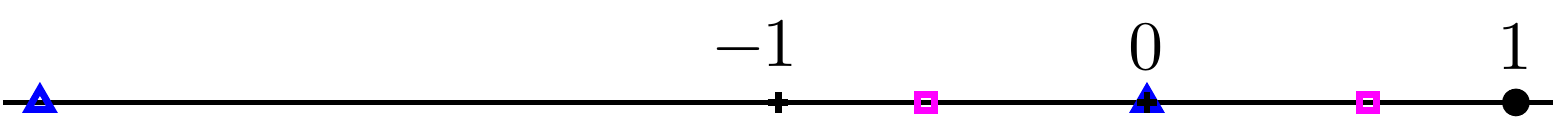}
\caption{Eigenvalues of $G_1(\theta)$ (magenta squares) and $M_1$ (blue triangles).}
\label{fig_thM1}
\end{figure}

Now we analyze the fixed point framework that employs $M_2$ and $b_2$, defined in 
\eqref{M1_M2} and \eqref{b1_b2}.

\subsection{Fixed Point Method based on $M_2$}

\begin{center}
\boxed{
\begin{minipage}{10.65cm}
\begin{algorithm}  {\bf (Primal-Dual Fixed Point Method; $M=M_2$)}
\label{alg_mfp2}
\end{algorithm}
\begin{tabbing}
\hspace*{7mm}\= \hspace*{7mm}\= \hspace*{7mm}\= \hspace*{7mm}\=
\hspace*{7mm}\= \hspace*{7mm}\= \kill
{\sc input}: $M=M_2$, $b=b_2$,
parameter $\theta>0$ \\
{\sc starting point}: $x^0\in\mathbb{R}^{d+N}$ \\ 
{\sc repeat for $k=0,1,2,\dots$}
\+ \\
{\sc set} $x^{k+1}=(1-\theta)x^k+\theta(Mx^k+b)$
\end{tabbing}
\end{minipage}
}
\end{center}

\begin{theorem}
\label{th_mfp2}
Let $x^0\in\mathbb{R}^{d+N}$ be an arbitrary starting point and consider the sequence $(x^k)_{k\in\mathbb{N}}$ 
generated by Algorithm~\ref{alg_mfp2} with 
$\theta\in\left(0,\dfrac{2\lambda n}{\lambda n+\sigma_1^2}\right)$. 
Then the sequence $(x^k)$ converges to the (unique) solution of the problem \eqref{pd_prob} 
at an asymptotic convergence rate of 
$
\rho_2(\theta)\stackrel{\rm def}{=}\sqrt{(1-\theta)^2+\dfrac{\theta^2\sigma_1^2}{\lambda n}}.
$ 
Furthermore, if we choose $\theta_2^*\stackrel{\rm def}{=}\dfrac{\lambda n}{\lambda n+\sigma_1^2}$, 
then the (theoretical) convergence rate is optimal and it is equal to 
$
\rho_2^*\stackrel{\rm def}{=}\dfrac{\sigma_1}{\sqrt{\lambda n+\sigma_1^2}}=
\sqrt{1-\theta_2^*}.
$ 
\end{theorem}
\begin{proof}
First, using Lemma \ref{pr_spectr}, we conclude that the eigenvalues of 
$G_2(\theta)\stackrel{\rm def}{=}(1-\theta)I+\theta M_2$ are 
$$
\left\{1-\theta\pm\dfrac{\theta\sigma_j}{\sqrt{\lambda n}}i\;,\quad j=1,\ldots,p
\right\}\cup\{1-\theta\},
$$
where $i=\sqrt{-1}$. The two ones with largest modulus are 
$1-\theta\pm\dfrac{\theta\sigma_1}{\sqrt{\lambda n}}i$ (see Figure \ref{fig_thM2}). 
So, the spectral radius of $G_2(\theta)$ is 
$$
\sqrt{(1-\theta)^2+\dfrac{\theta^2\sigma_1^2}{\lambda n}}=\rho_2(\theta).
$$ 
Further, the values of $\theta$ for which the eigenvalues of $G_2(\theta)$ have modulus 
less than one can be found by solving 
$(1-\theta)^2+\dfrac{\theta^2\sigma_1^2}{\lambda n}<1$ giving 
$$
0<\theta<\dfrac{2\lambda n}{\lambda n+\sigma_1^2}.
$$
The asymptotic convergence follows from the fact that 
$\|G_2(\theta)^k\|^{1/k}\to\rho_2(\theta)$. Indeed, using \eqref{distxk} we conclude that 
$$
\left(\dfrac{\|x^k-x^*\|}{\|x^0-x^*\|}\right)^{1/k}\leq \|G_2(\theta)^k\|^{1/k}\to\rho_2(\theta).
$$
This means that given $\gamma>0$, there exists $k_0\in\mathbb{N}$ such that 
$$
\|x^k-x^*\|\leq(\rho_2(\theta)+\gamma)^k\|x^0-x^*\|
$$
for all $k\geq k_0$. Finally, the optimal parameter $\theta_2^*$ and the corresponding 
optimal rate $\rho_2^*$ can be obtained directly by solving 
$$
\min_{\theta>0}\; (1-\theta)^2+\dfrac{\theta^2\sigma_1^2}{\lambda n}.
$$
\end{proof}

\medskip

The left picture of Figure \ref{fig_thM2} illustrates, in the complex plane, the eigenvalues of 
$G_2(\theta)$ (magenta squares) together with the eigenvalues of $M_2$ (blue triangles), for a fixed 
value of the parameter $\theta$. On the right we show, for each $\theta\in(0,1)$, 
{one of the two eigenvalues of $G_2(\theta)$ farthest from the origin.} 
The dashed segment corresponds to the admissible values for $\theta$, that is, the 
eigenvalues with modulus less than one. The square corresponds to the optimal choice of 
$\theta$. 

\begin{figure}[htbp]
\centering
\includegraphics[scale=0.42]{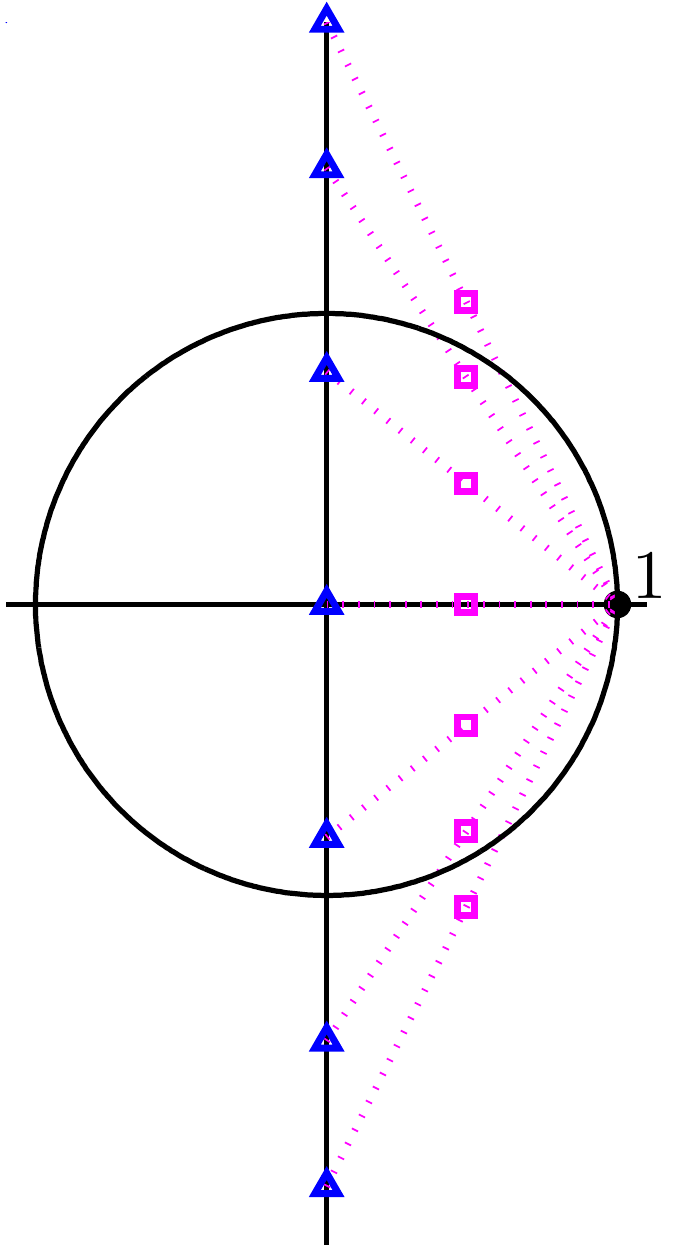}
\hspace{1.25cm}
\includegraphics[scale=0.42]{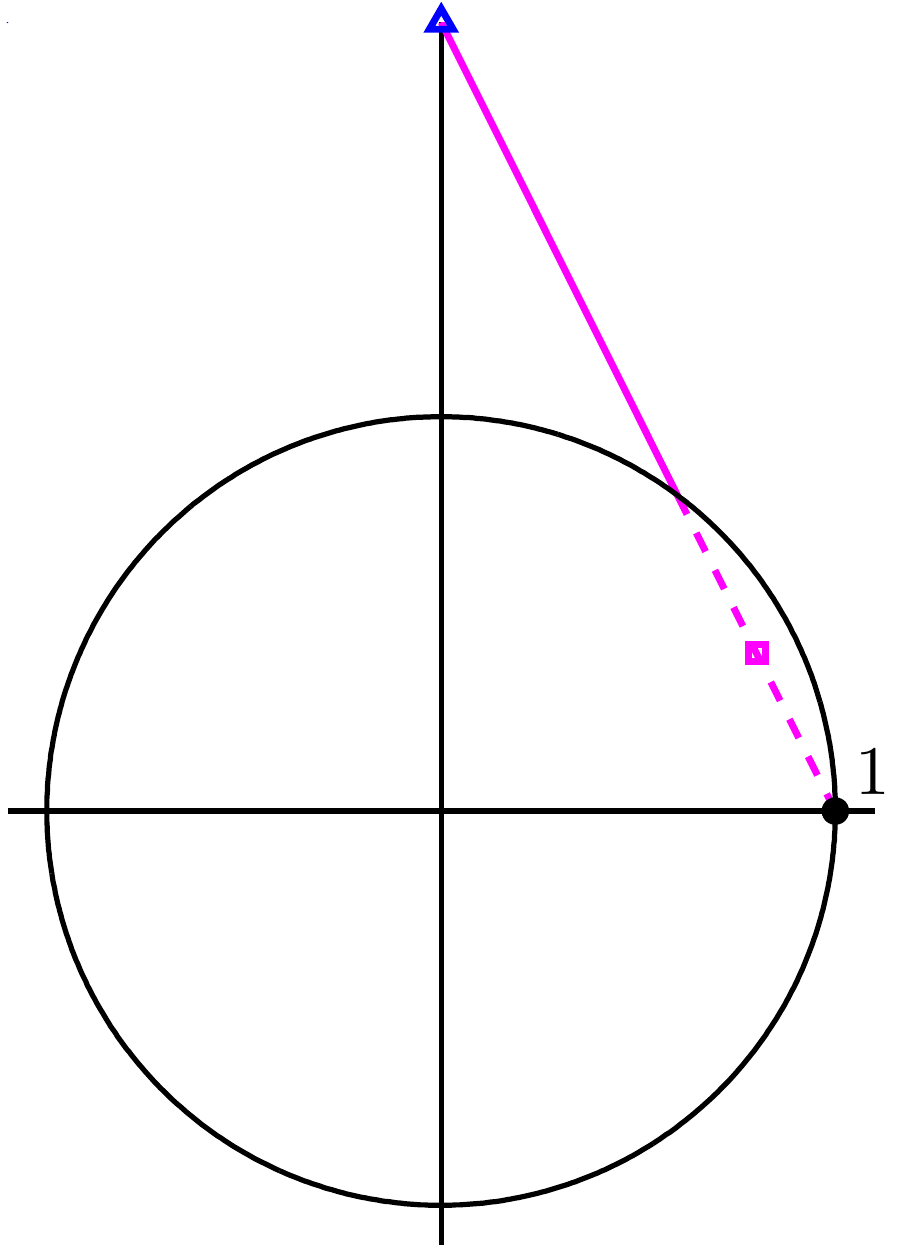}
\caption{Eigenvalues of $G_2(\theta)$ (magenta squares) and $M_2$ (blue triangles), 
represented in the complex plane.}
\label{fig_thM2}
\end{figure}

\subsection{Comparison of the rates}
We summarize the discussion above in Table \ref{table1} which brings the comparison between the 
pure ($\theta=1$) and optimal ($\theta=\theta_j^*$, $j=1,2$) versions of Algorithm~\ref{alg_mfp}. 
We can see that the convergence rate of the optimal version is 
$\lambda n/(2\lambda n+\sigma_1^2)$ times the one of the pure version if $M_1$ is 
employed \redcolor{(Algorithm \ref{alg_mfp1})} and $\sqrt{\lambda n/(\lambda n+\sigma_1^2)}$ 
times the pure version when using $M_2$ \redcolor{(Algorithm \ref{alg_mfp2})}. Moreover, in any case, 
employing $M_1$ provides faster convergence. This can be seen in Figure \ref{fig_fxp}, where 
Algorithm \ref{alg_mfp} was applied to solve the problem \eqref{pd_prob}. The dimensions considered 
were $d=200$, $m=1$ and $n=5000$ (so that the total dimension is $d+N=d+nm=5200$). 

We also mention that the pure version does not require the knowledge of $\sigma_1$, but it 
may not converge. On the other hand, the optimal version always converges, but 
$\theta$ depends on $\sigma_1$. 

\begin{table}[htbp]
\begin{center}
\renewcommand{\arraystretch}{2.25}
\begin{tabular}{|c|c|c|}
\hline
 & \redcolor{PDFP1($\theta$)} & \redcolor{PDFP2($\theta$)}  \\
\cline{1-3}
Range of $\theta$ & 
$\left(0,\dfrac{2\lambda n}{\lambda n+\sigma_1^2}\right)$ & 
$\left(0,\dfrac{2\lambda n}{\lambda n+\sigma_1^2}\right)$   \\
\cline{1-3}
Pure ($\theta=1$)  & $\dfrac{\sigma_1^2}{\lambda n}$ & 
$\dfrac{\sigma_1}{\sqrt{\lambda n}}$   \\
\cline{1-3}
Optimal ($\theta=\theta_j^*$)  & $\dfrac{\sigma_1^2}{2\lambda n+\sigma_1^2}=1-\theta_1^*$ & 
$\sqrt{\dfrac{\sigma_1^2}{\lambda n+\sigma_1^2}}=\sqrt{1-\theta_2^*}$   \\
\hline
\end{tabular}
\end{center}
\caption{\redcolor{Ranges of convergence and convergence rates of pure and optimal versions 
of Algorithm \ref{alg_mfp1} (the one that uses $M_1$), indicated by PDFP1($\theta$), and 
Algorithm \ref{alg_mfp2} (the one that uses $M_2$), PDFP2($\theta$).}}
\label{table1}
\end{table}

\begin{figure}[htbp]
\centering
\includegraphics[scale=0.45]{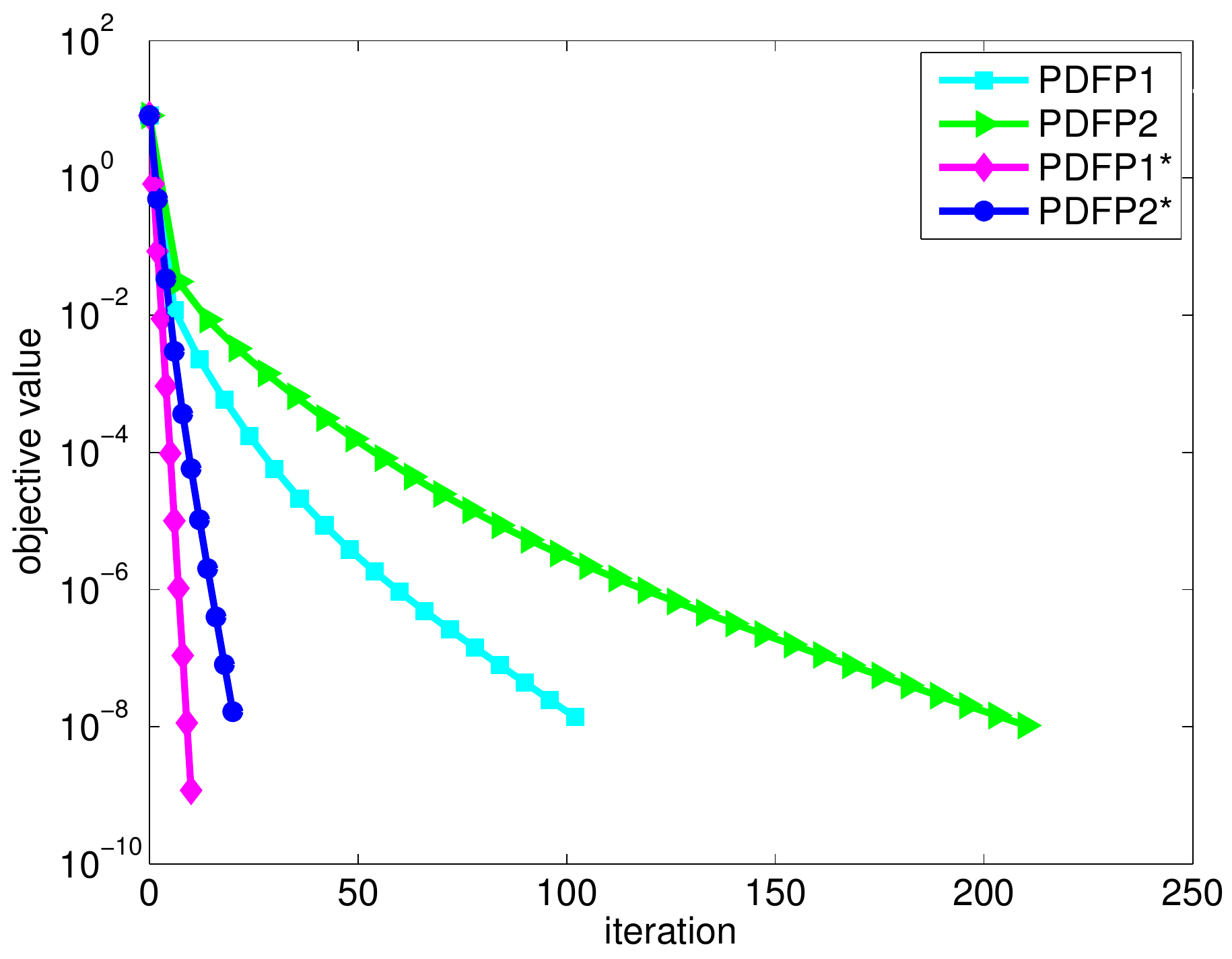}
\caption{\redcolor{Performance of pure and optimal versions of Algorithms~\ref{alg_mfp1} and 
\ref{alg_mfp2} applied to solve the problem \eqref{pd_prob}. The picture shows the objective values 
against the number of iterations. The dimensions considered were 
$d=200$, $m=1$ and $n=5000$ (so that the total dimension is $d+N=d+nm=5200$). The matrix 
$A\in\mathbb{R}^{d\times N}$ and the vector $y\in\mathbb{R}^{N}$ were randomly generated. 
For simplicity of notation we have denoted the pure and optimal versions of Algorithm~\ref{alg_mfp1}
by PDFP1 and PDFP1*, respectively. Analogously, for Algorithm~\ref{alg_mfp2}, we used 
PDFP2 and PDFP2* to denote the pure and optimal versions, respectively.}}
\label{fig_fxp}
\end{figure}

\subsection{Direct relationship between the iterates of the two methods}
Another relation regarding the employment of $M_1$ or $M_2$ in the pure version of 
Algorithm~\ref{alg_mfp}, which is also illustrated in Figure \ref{fig_fxp}, is that 
one step of the method with $M_1$ corresponds exactly to two steps of the one with $M_2$. 
Indeed, note first that $M_2^2=M_1$. Thus, denoting the current point by $x$ and the next 
iterate by $x_M^+$, in view of \eqref{mfp1_mfp2} we have 
\begin{align*}
x_{M_2}^{++} & =  M_2x_{M_2}^++b_2 =
M_2(M_2x+b_2)+b_2 =
 M_1x+M_2b_2+b_2 \\
& =  M_1x+b_1
=  x_{M_1}^+ .
\end{align*}

In Section \ref{sec_qtz} we shall see how this behavior can invert with a small change in the 
computation of the dual variable.

\subsection{Complexity results}
\label{sec_cpx}
In order to establish the complexity of Algorithm~\ref{alg_mfp} we need to 
calculate the condition number of the objective function, defined in \eqref{pd_prob}. 
Note that the Hessian of $f$ is given by 
$$
\nabla^2f=\dfrac{1}{n}\left(\begin{array}{cc} AA^T+\lambda nI & 0 \\ 
0 & \frac{1}{\lambda n}A^TA+I \end{array}\right)
$$
Let us consider two cases:
\begin{itemize}
\item If $\lambda n\geq 1$, then 
$\sigma_1^2+\lambda n\geq\sigma_1^2+1\geq\dfrac{\sigma_1^2}{\lambda n}+1$, 
which in turn implies that the largest eigenvalue of $\nabla^2f$ is 
$L=\dfrac{\sigma_1^2+\lambda n}{n}$. 
The smallest eigenvalue is 
$$
\left\{\begin{array}{l}
\dfrac{1}{n},\, \mbox{if } d<N  \vspace{.152cm} 
\\
\dfrac{\sigma_d^2}{\lambda n^2}+\dfrac{1}{n},\, \mbox{if } d=N  \vspace{.152cm} 
\\
\min\left\{\lambda,\dfrac{\sigma_N^2}{\lambda n^2}+\dfrac{1}{n}\right\}, \, \mbox{if } d>N. \vspace{.152cm} 
\end{array}\right.
$$
Therefore, if $d<N$, the condition number of $\nabla^2f$ is 
the condition number of $\nabla^2f$ is 
\begin{equation}
\label{cond_f1}
\sigma_1^2+\lambda n.
\end{equation}

\item If $\lambda n<1$, then 
$\sigma_1^2+\lambda n<\sigma_1^2+1<\dfrac{\sigma_1^2}{\lambda n}+1$, 
which in turn implies that the largest eigenvalue of $\nabla^2f$ is 
$L=\dfrac{\sigma_1^2+\lambda n}{\lambda n^2}$. The smallest eigenvalue is 
$$
\left\{\begin{array}{l}
\min\left\{\dfrac{\sigma_d^2}{n}+\lambda,\dfrac{1}{n}\right\}, \, \mbox{if } d<N \vspace{.152cm} \\
\dfrac{\sigma_d^2}{n}+\lambda,\, \mbox{if } d=N  \vspace{.152cm} \\
\lambda,\, \mbox{if } d>N. 
\end{array}\right.
$$

So, assuming that $d<N$, the condition number is 
$\dfrac{\sigma_1^2+\lambda n}{\lambda n^2\min\left\{\dfrac{\sigma_d^2}{n}+\lambda,\dfrac{1}{n}\right\}}$.
If $A$ is rank deficient, then the condition number is 
\begin{equation}
\label{cond_f2}
\dfrac{\sigma_1^2+\lambda n}{(\lambda n)^2}.
\end{equation}
\end{itemize}

We stress that despite the analysis was made in terms of the sequence $x^k=(w^k,\alpha^k)$, 
the linear convergence also applies to objective values. Indeed, since $f$ is $L$-smooth, 
we have 
$$
f(x^k)\leq f(x^*)+\nabla f(x^*)^T(x^k-x^*)+\dfrac{L}{2}\|x^k-x^*\|^2=\dfrac{L}{2}\|x^k-x^*\|^2,
$$
where the equality follows from the fact that the optimal objective value is zero. 
Therefore, if we want to get $f(x^k)-f(x^*)<\varepsilon$ and we have linear convergence 
rate $\rho$ on the sequence $(x^k)$, then it is enough to enforce 
$$
\dfrac{L}{2}\rho^{2k}\|x^0-x^*\|^2<\varepsilon,
$$
or equivalently, 
\begin{equation}
\label{cpx0}
k>\dfrac{-1}{2\log\rho}\log\left(\dfrac{\|x^0-x^*\|^2L}{2\varepsilon}\right).
\end{equation}
Using the estimate $\log(1-\theta)\approx-\theta$, we can approximate the \redcolor{right hand} side of 
\eqref{cpx0} by 
\begin{equation}
\label{cpx0ap1}
\dfrac{1}{2\theta_1^*}\log\left(\dfrac{\|x^0-x^*\|^2L}{2\varepsilon}\right),
\end{equation}
in the case $M_1$ is used and by 
\begin{equation}
\label{cpx0ap2}
\dfrac{1}{\theta_2^*}\log\left(\dfrac{\|x^0-x^*\|^2L}{2\varepsilon}\right),
\end{equation}
if we use $M_2$.

In order to estimate the above expressions in terms of the condition number, let us consider 
the more common case $\lambda n\geq 1$. Then the condition number of 
$\nabla^2f$ is given by \eqref{cond_f1}, that is, 
\begin{equation}
\label{cond_f}
\kappa\stackrel{\rm def}{=}\sigma_1^2+\lambda n.
\end{equation}
So, if we use $M_1$, the complexity is proportional to 
\begin{equation}
\label{cpx1}
\dfrac{1}{2\theta_1^*}=\dfrac{\sigma_1^2+2\lambda n}{4\lambda n}=
\dfrac{\kappa+\lambda n}{4\lambda n}.
\end{equation}
If we use $M_2$, the complexity is proportional to 
\begin{equation}
\label{cpx2}
\dfrac{1}{\theta_2^*}=\dfrac{\lambda n+\sigma_1^2}{\lambda n}=
\dfrac{\kappa}{\lambda n}.
\end{equation}

\section{Accelerated Primal-Dual Fixed Point Method}
\label{sec_qtz}
Now we present our main contribution. When we employ Algorithm \ref{alg_mfp2}, 
the primal and dual variables are mixed in two equations. More precisely, in view of 
\eqref{alphabar} the iteration in this case can be rewritten as 
$$
\left\{
\begin{array}{l}
w^{k+1}=(1-\theta)w^k+\theta\bar\alpha^k \\
\alpha^{k+1}=(1-\theta)\alpha^k+\theta(y-A^Tw^k).
\end{array}
\right.
$$
The idea here is to apply block Gauss-Seidel to this system. That is, we use the freshest $w$ 
to update $\alpha$. Let us state formally the method by means of the following framework.

\begin{center}
\boxed{
\begin{minipage}{9.6cm}
\begin{algorithm} {\bf Accelerated Fixed Point Method}
\label{alg_qtz}
\end{algorithm}
\begin{tabbing}
\hspace*{7mm}\= \hspace*{7mm}\= \hspace*{7mm}\= \hspace*{7mm}\=
\hspace*{7mm}\= \hspace*{7mm}\= \kill
{\sc input}: matrix $A\in\mathbb{R}^{d\times N}$, vector $y\in\mathbb{R}^N$, 
parameter $\theta\in(0,1]$ \\ 
{\sc starting points}:  $w^0\in\mathbb{R}^d$ and $\alpha^0\in\mathbb{R}^N$ \\ 
{\sc repeat for $k=0,1,2,\dots$}
\+ \\
{\sc set} $w^{k+1}=(1-\theta)w^k+\theta\bar\alpha^k$ \\
{\sc set} $\alpha^{k+1}=(1-\theta)\alpha^k+\theta(y-A^Tw^{k+1})$ 
\end{tabbing}
\end{minipage}
}
\end{center}

Due to this modification, we can achieve faster convergence. 
This algorithm is a deterministic version of a randomized primal-dual algorithm (Quartz) 
proposed and analyzed by Qu, Richt\'arik and Zhang \cite{Quartz}.  

\subsection{Convergence analysis}
\label{sec_cvqtz}
In this section we study the convergence of Algorithm \ref{alg_qtz}. We shall 
determine all values for the parameter $\theta$ for which this algorithm converges as 
well as the one giving the best convergence rate. 

To this end, we start by showing that Algorithm \ref{alg_qtz} can be viewed as a fixed 
point scheme. Then we determine the ``dynamic'' spectral properties of the associated matrices, 
which are parameterized by $\theta$. 

First, note that the iteration of our algorithm can be written as
$$
\left(\begin{array}{cc} I & 0 \\ \theta A^T & I \end{array}\right)
\left(\begin{array}{c} w^{k+1} \\ \alpha^{k+1}\end{array}\right)=
\left(\begin{array}{cc} (1-\theta)I & \frac{\theta}{\lambda n}A  \vspace{.15cm}\\ 
0 & (1-\theta)I \end{array}\right)
\left(\begin{array}{c} w^k \\ \alpha^k\end{array}\right)+
\left(\begin{array}{c} 0 \\ \theta y \end{array}\right)
$$
or in a compact way as 
\begin{equation}
\label{qtzfixp}
x^{k+1}=G_3(\theta)x^k+f
\end{equation}
with 
% \begin{equation}
% \label{Gqtz}
% G_3(\theta)=\left(\begin{array}{cc} I & 0 \\ \theta A^T & I \end{array}\right)^{-1}
% \left(\begin{array}{cc} (1-\theta)I & \frac{\theta}{\lambda n}A  \vspace{.15cm}\\ 
% 0 & (1-\theta)I \end{array}\right)=
% (1-\theta)I+\theta\left(\begin{array}{cc} 0 & \frac{1}{\lambda n}A \vspace{.15cm}\\ 
% (\theta-1)A^T & -\frac{\theta}{\lambda n}A^TA \end{array}\right)
% \end{equation}
% and 
% $$
% f=\left(\begin{array}{cc} I & 0 \\ \theta A^T & I \end{array}\right)^{-1}
% \left(\begin{array}{c} 0 \\ \theta y \end{array}\right)=
% \left(\begin{array}{c} 0 \\ \theta y \end{array}\right).
% $$
\begin{equation}
\label{Gqtz}
G_3(\theta)=
(1-\theta)I+\theta\left(\begin{array}{cc} 0 & \frac{1}{\lambda n}A \vspace{.15cm}\\ 
(\theta-1)A^T & -\frac{\theta}{\lambda n}A^TA \end{array}\right)
\end{equation}
and 
$
f=\left(\begin{array}{c} 0 \\ \theta y \end{array}\right).
$

We know that if the spectral radius of $G_3(\theta)$ is less that $1$, 
then the sequence defined by \eqref{qtzfixp} converges. Indeed, in this case the limit point 
is just $x^*$, the solution of the problem \eqref{pd_prob}. This follows from the fact that 
$x^*=G_3(\theta)x^*+f$.

Next lemma provides the spectrum of $G_3(\theta)$. 
\begin{lemma}
\label{lm_eigGqtz}
The eigenvalues of the matrix $G_3(\theta)$, defined in \eqref{Gqtz}, are given by 
$$
\dfrac{1}{2\lambda n}
\left\{2(1-\theta)\lambda n-\theta^2\sigma_j^2\pm
\theta\sigma_j\sqrt{\theta^2\sigma_j^2-4(1-\theta)\lambda n}
\;,\quad j=1,\ldots,p \right\}\cup\{1-\theta\}.
$$
\end{lemma}
\begin{proof}
Consider the matrix 
$
M_3(\theta)\stackrel{\rm def}{=}
\left(\begin{array}{cc} 0 & \frac{1}{\lambda n}A \vspace{.15cm}\\ 
(\theta-1)A^T & -\frac{\theta}{\lambda n}A^TA \end{array}\right).
$
Using the singular value decomposition of $A$, given in \eqref{svdA}, we can write 
$$
M_3(\theta)=
\left(\begin{array}{cc} U & 0 \\ 0 & V \end{array}\right)
\left(\begin{array}{cc} 0 & \frac{1}{\lambda n}\Sigma \vspace{.15cm}\\ 
(\theta-1)\Sigma^T & -\frac{\theta}{\lambda n}\Sigma^T\Sigma \end{array}\right)
\left(\begin{array}{cc} U^T & 0 \\ 0 & V^T \end{array}\right).
$$
Therefore, the eigenvalues of $M_3(\theta)$ are the same as the ones of 
\begin{eqnarray}
& \left(\begin{array}{cc} 0 & \frac{1}{\lambda n}\Sigma \vspace{.15cm}\\ 
(\theta-1)\Sigma^T & -\frac{\theta}{\lambda n}\Sigma^T\Sigma \end{array}\right)
= 
\left(\begin{array}{ccccccc} 
0 & & 0 & &   -c\widetilde\Sigma & & 0 \\ 
0 & & 0 & &  0  & & 0  \\ 
(\theta-1)\widetilde\Sigma & & 0 & &  \theta c\widetilde\Sigma^2 & & 0 \\ 
0 & & 0 & &  0 & & 0 
\end{array}\right) 
& \hspace{-.5cm}\begin{array}{c} p \\ d-p \\ p \\ N-p \end{array}  \nonumber \\
& \begin{array}{ccccc}
\hspace{5.28cm} &  p & \hspace{.6cm} d-p & \hspace{.35cm} p  & \hspace{.125cm}N-p   \nonumber \\
\end{array}
\end{eqnarray}
where $c=-\dfrac{1}{\lambda n}$ and $\widetilde\Sigma$ is defined in \eqref{Sigma_A}. 
The characteristic polynomial of this matrix is 
$$
\begin{array}{rcl}
p_{\theta}(t)&=&\det\left(\begin{array}{ccccccc} 
tI & & 0 & &   c\widetilde\Sigma & &  0 \\ 0 & & tI & &  0  & & 0  \\ 
(1-\theta)\widetilde\Sigma & & 0 & &  tI-\theta c\widetilde\Sigma^2 & & 0 
\\ 0 & & 0 & &  0 & & tI 
\end{array}\right)
\\& =&\det\left(\begin{array}{cccc} 
tI & c\widetilde\Sigma & 0 & 0 \\ (1-\theta)\widetilde\Sigma & tI-\theta c\widetilde\Sigma^2 &  0  & 0  \\ 
0 & 0 &  tI & 0 \\ 0 & 0 &  0 & tI 
\end{array}\right). 
\end{array}
$$
Using Proposition \ref{pr_det} and denoting $q=N+d-2p$, we obtain 
$$
p_{\theta}(t)=t^{q}\prod_{j=1}^{p}\Big(t(t-\theta c\sigma_j^2)-c(1-\theta)\sigma_j^2\Big)=
t^{q}\prod_{j=1}^{p}\left(t^2+\dfrac{\theta\sigma_j^2}{\lambda n}t+
\dfrac{(1-\theta)\sigma_j^2}{\lambda n}\right).
$$
Thus, the eigenvalues of $M_3(\theta)$ are 
$$
\dfrac{1}{2\lambda n}
\left\{-\theta\sigma_j^2\pm
\sigma_j\sqrt{\theta^2\sigma_j^2-4(1-\theta)\lambda n}
\;,\quad j=1,\ldots,p \right\}\cup\{0\},
$$
so that the eigenvalues of $G_3(\theta)=(1-\theta)I+\theta M_3(\theta)$ are 
$$
\dfrac{1}{2\lambda n}
\left\{2(1-\theta)\lambda n-\theta^2\sigma_j^2\pm
\theta\sigma_j\sqrt{\theta^2\sigma_j^2-4(1-\theta)\lambda n}
\;,\quad j=1,\ldots,p \right\}\cup\{1-\theta\},
$$
giving the desired result.
\end{proof}

\medskip

Figure \ref{fig_eigqtz1} illustrates, in the complex plane, the spectrum of the matrix 
$G_3(\theta)$ for many different values of $\theta$. We used $n=250$, $d=13$, $m=1$ 
(therefore $N=250$), $\lambda=0.3$ and a random matrix $A\in\mathbb{R}^{d\times N}$. 
The pictures point out the fact that for some range of $\theta$ the spectrum is contained in 
a circle and for other values of $\theta$ some of the eigenvalues remain in a circle while others 
are distributed along the real line, moving monotonically as this parameter changes. These 
statements will be proved in the sequel. 

\begin{figure}[htbp]
\centering
\includegraphics[scale=0.248]{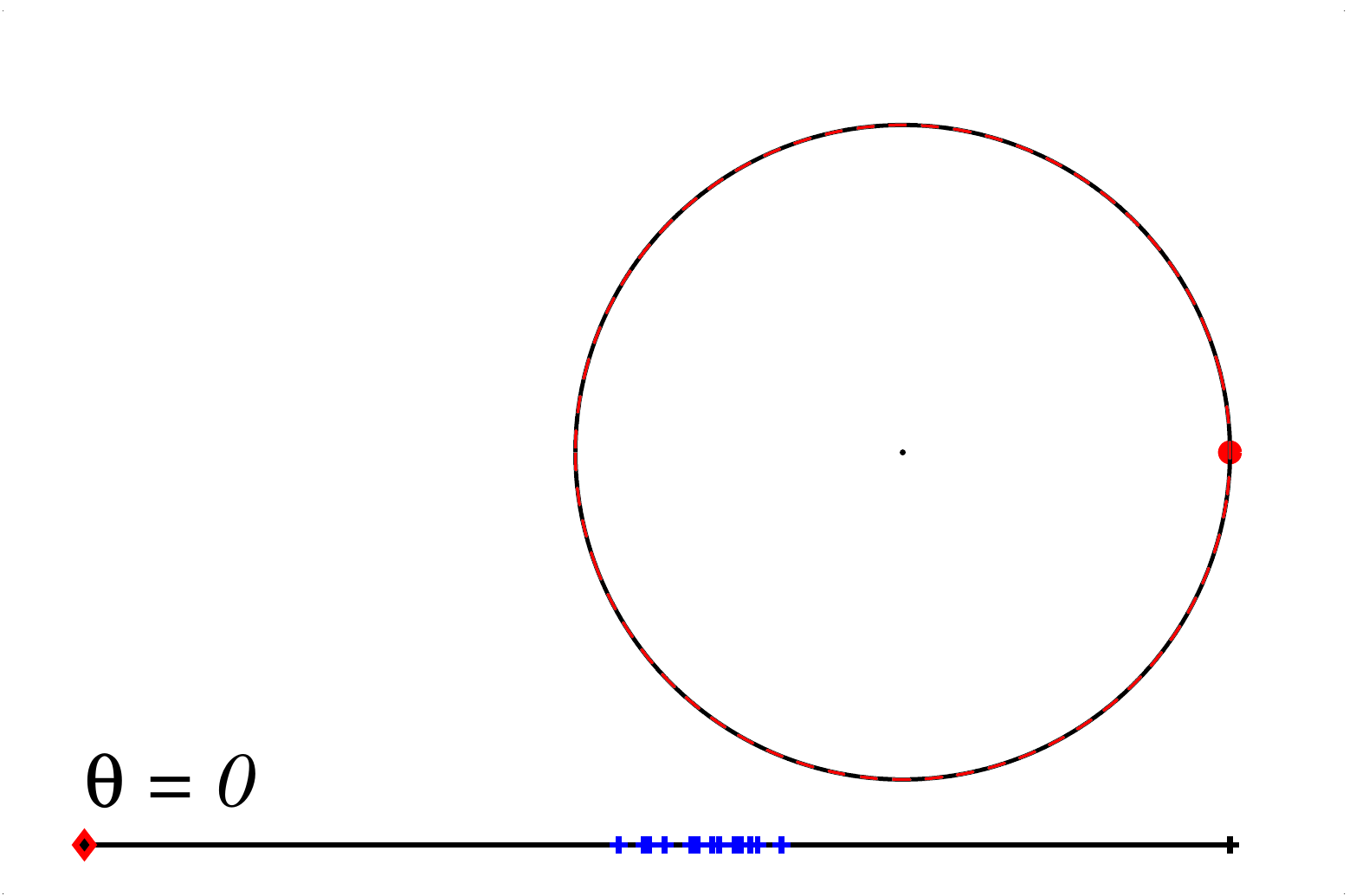}
\includegraphics[scale=0.248]{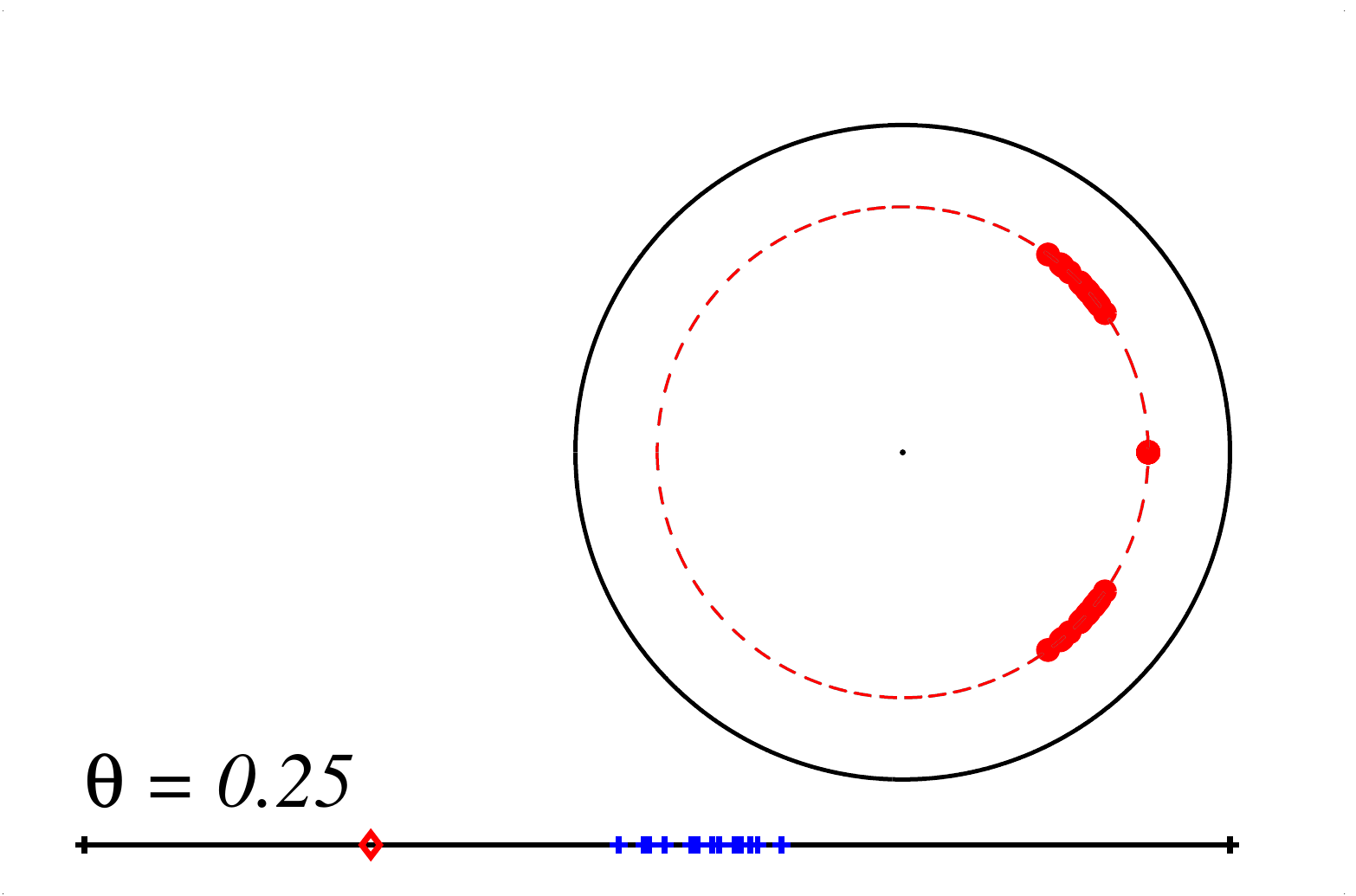}
\includegraphics[scale=0.248]{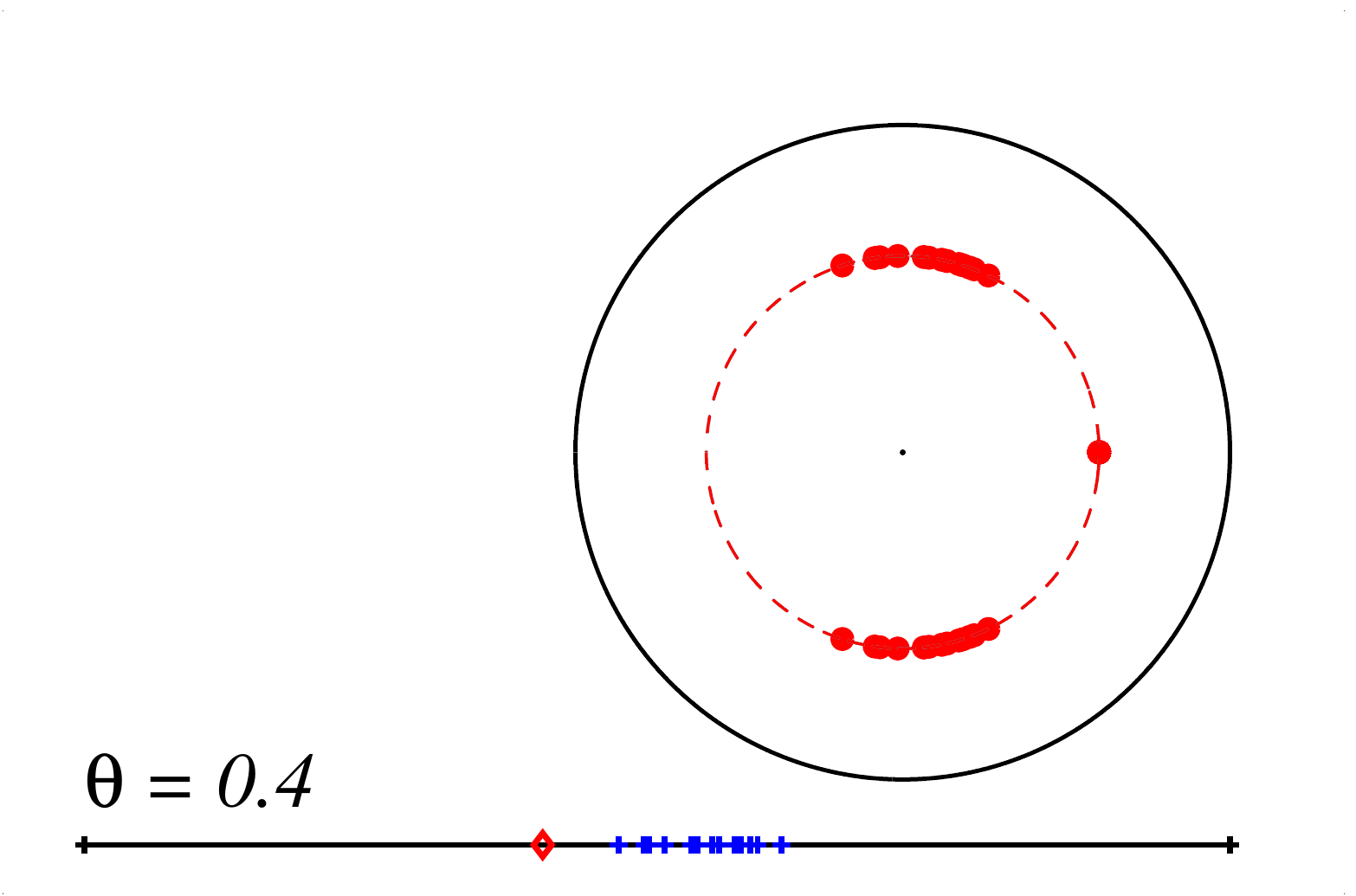}
\includegraphics[scale=0.248]{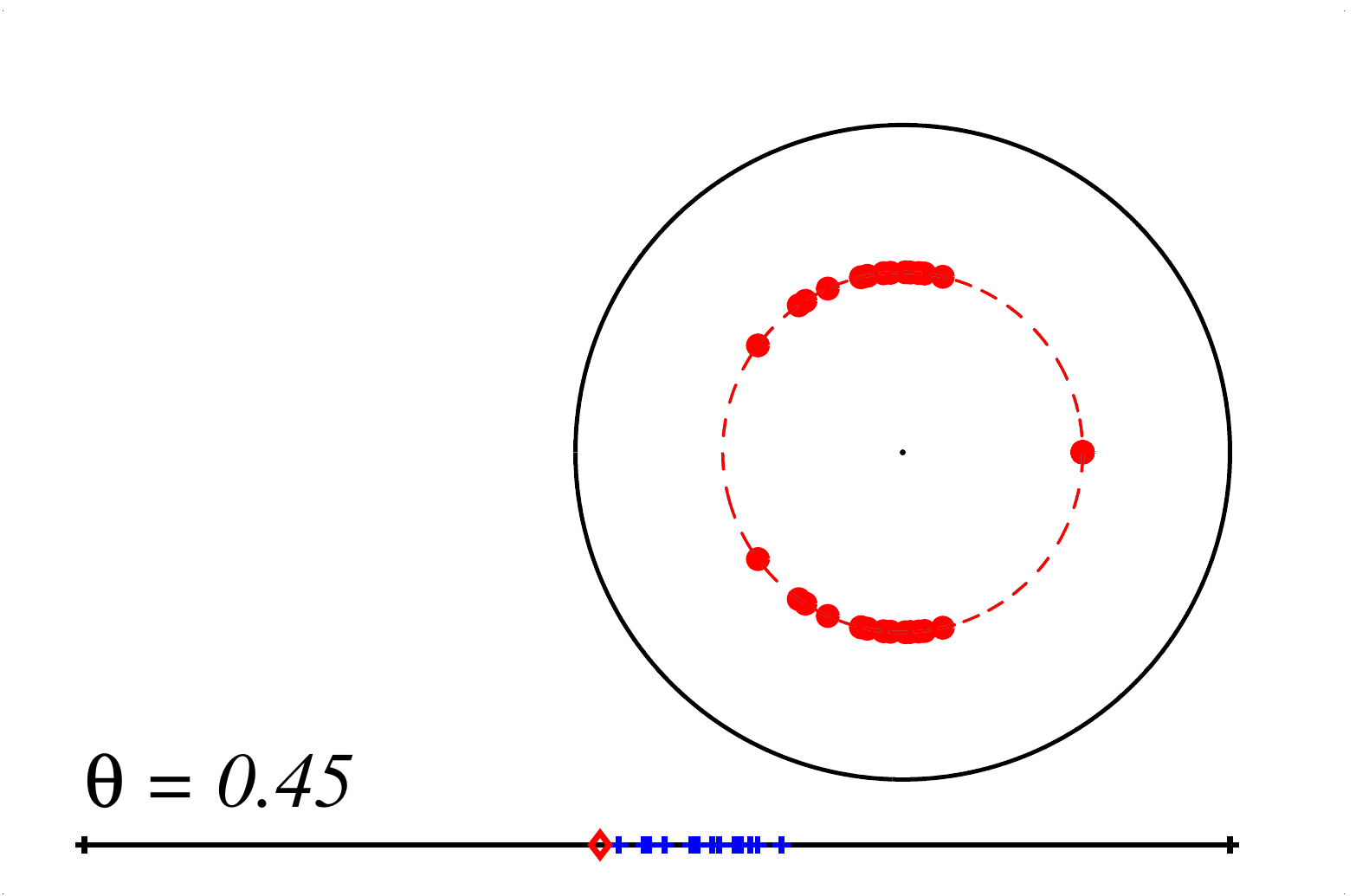}
\includegraphics[scale=0.248]{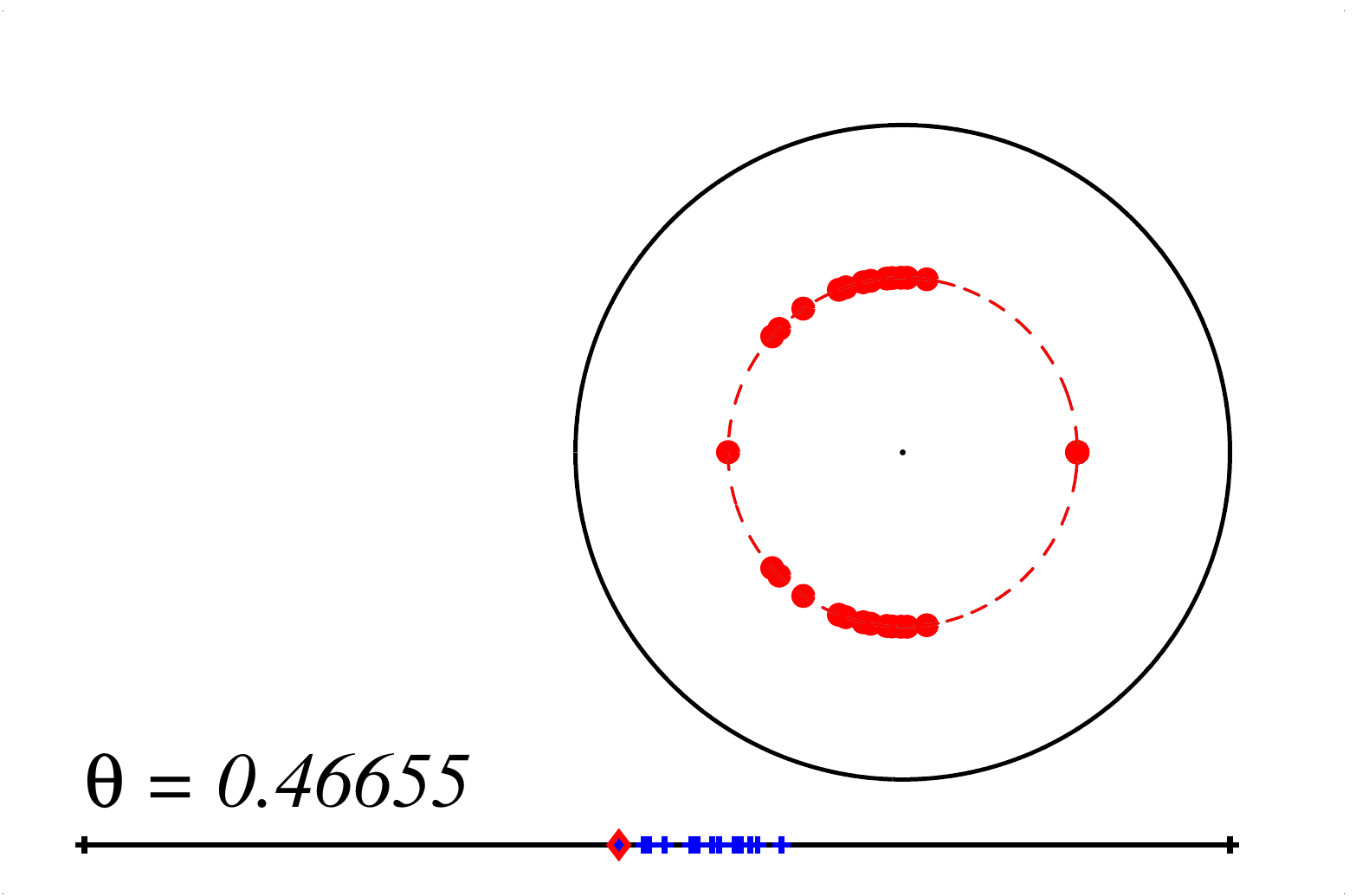}
\includegraphics[scale=0.248]{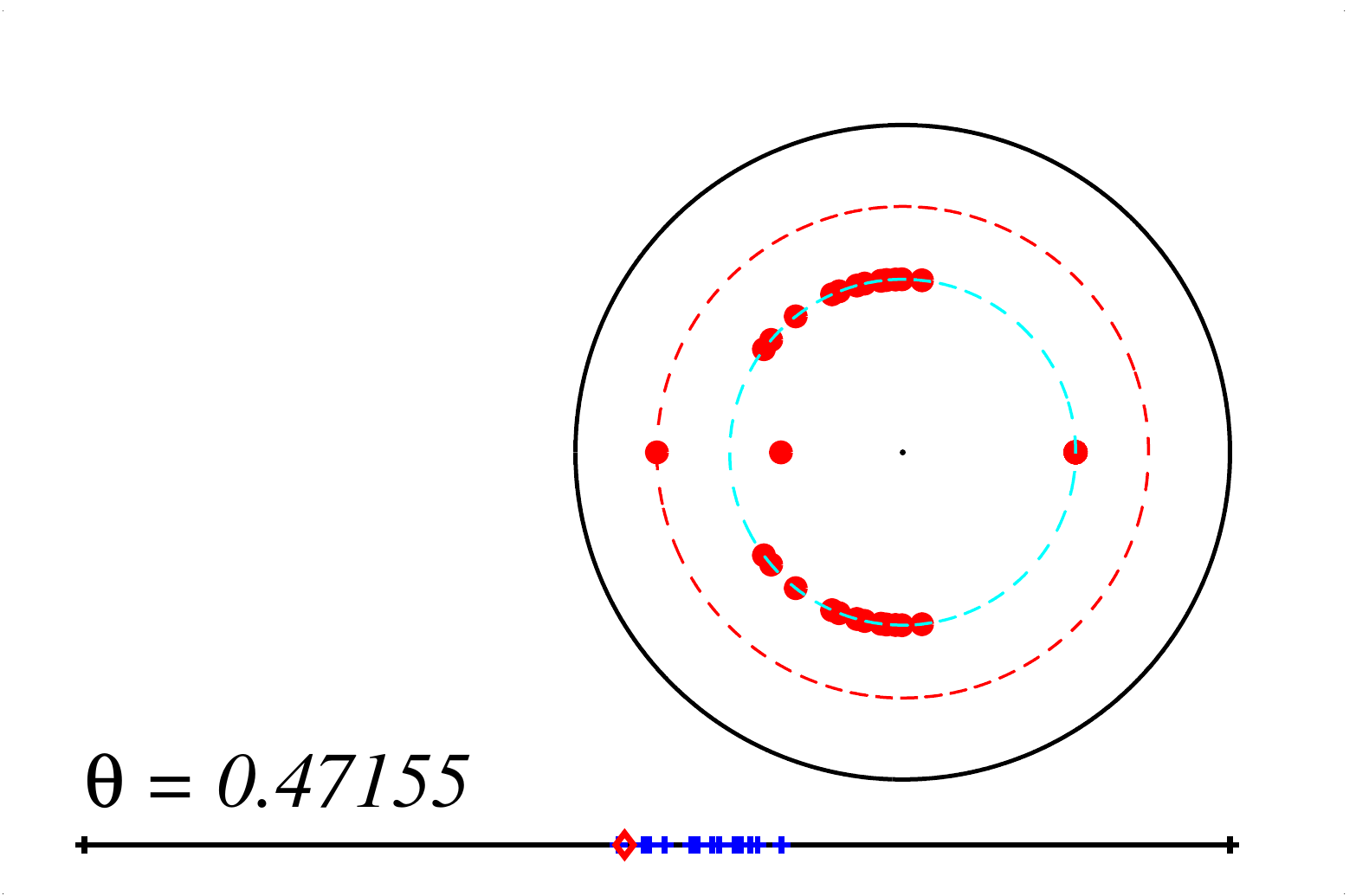}
\includegraphics[scale=0.248]{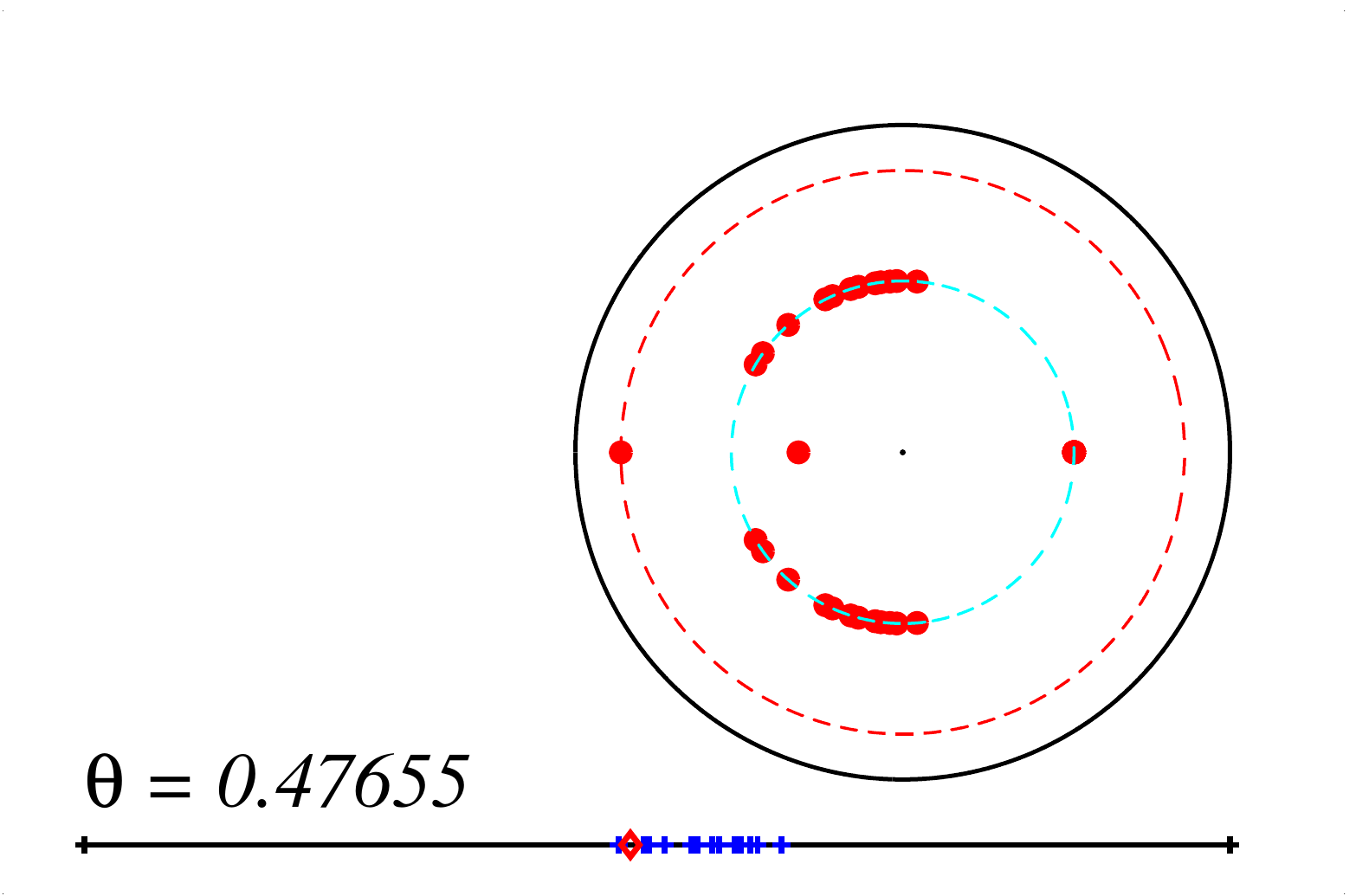}
\includegraphics[scale=0.248]{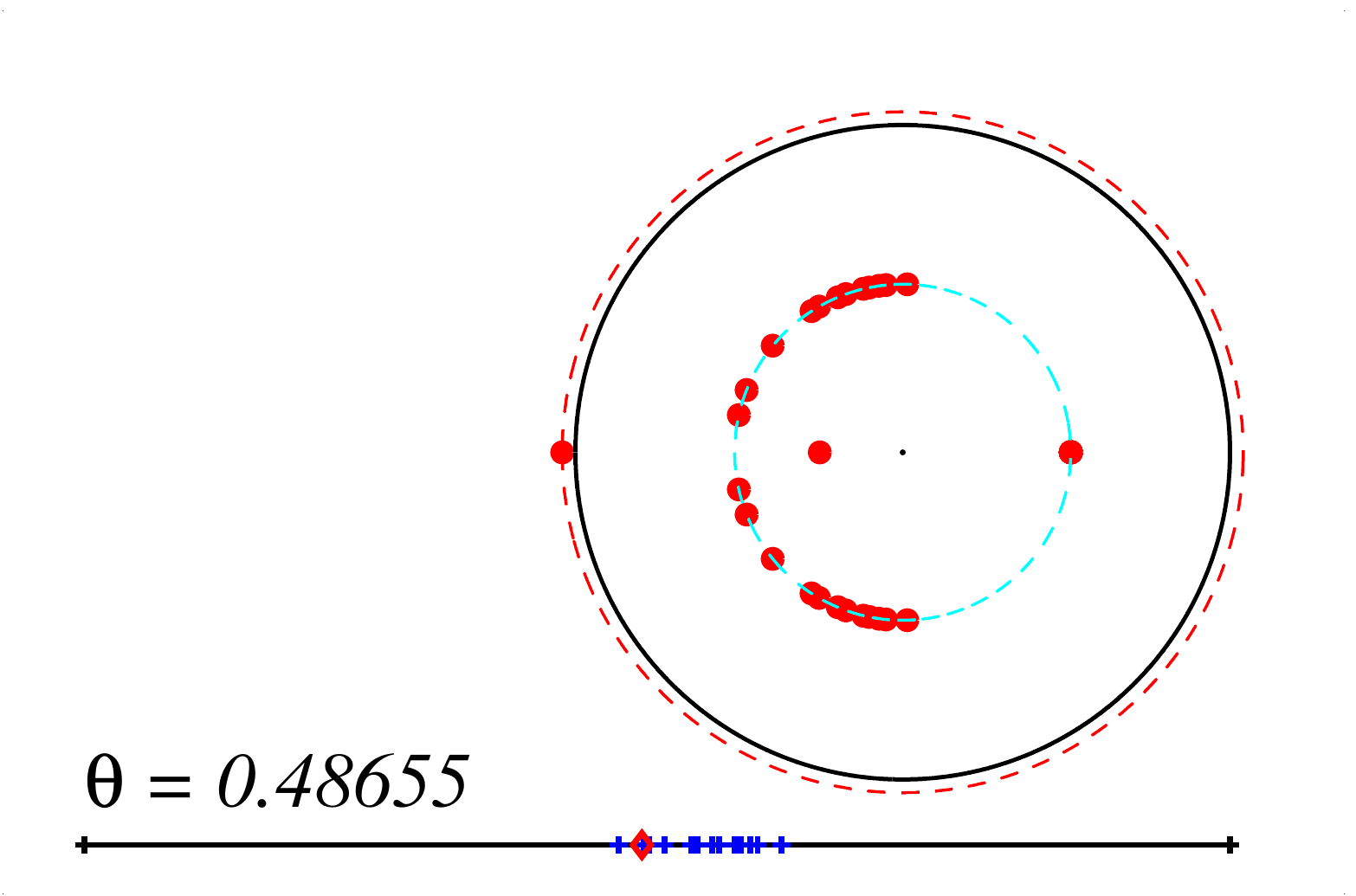}
\includegraphics[scale=0.248]{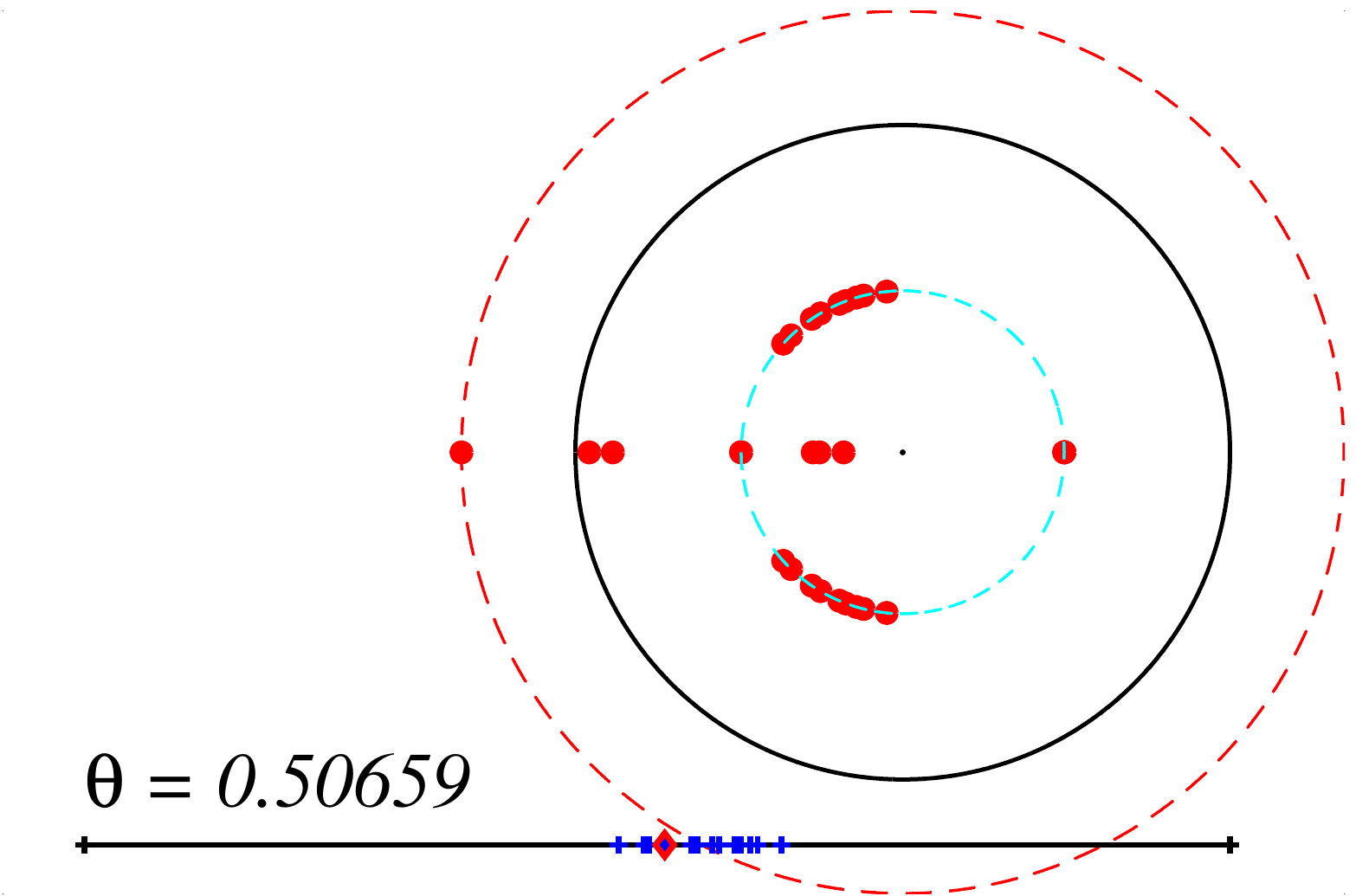}
\includegraphics[scale=0.248]{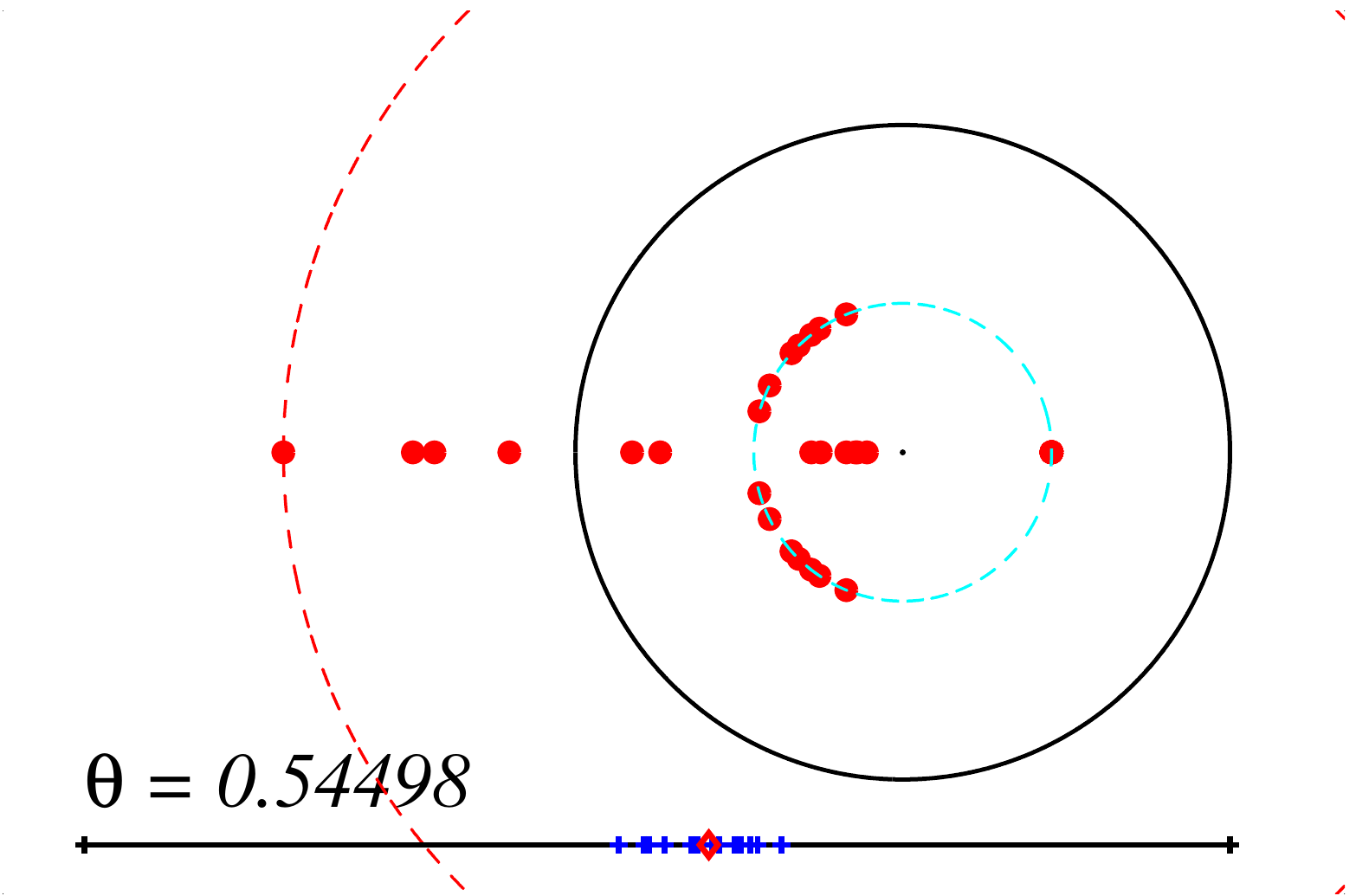}
\includegraphics[scale=0.248]{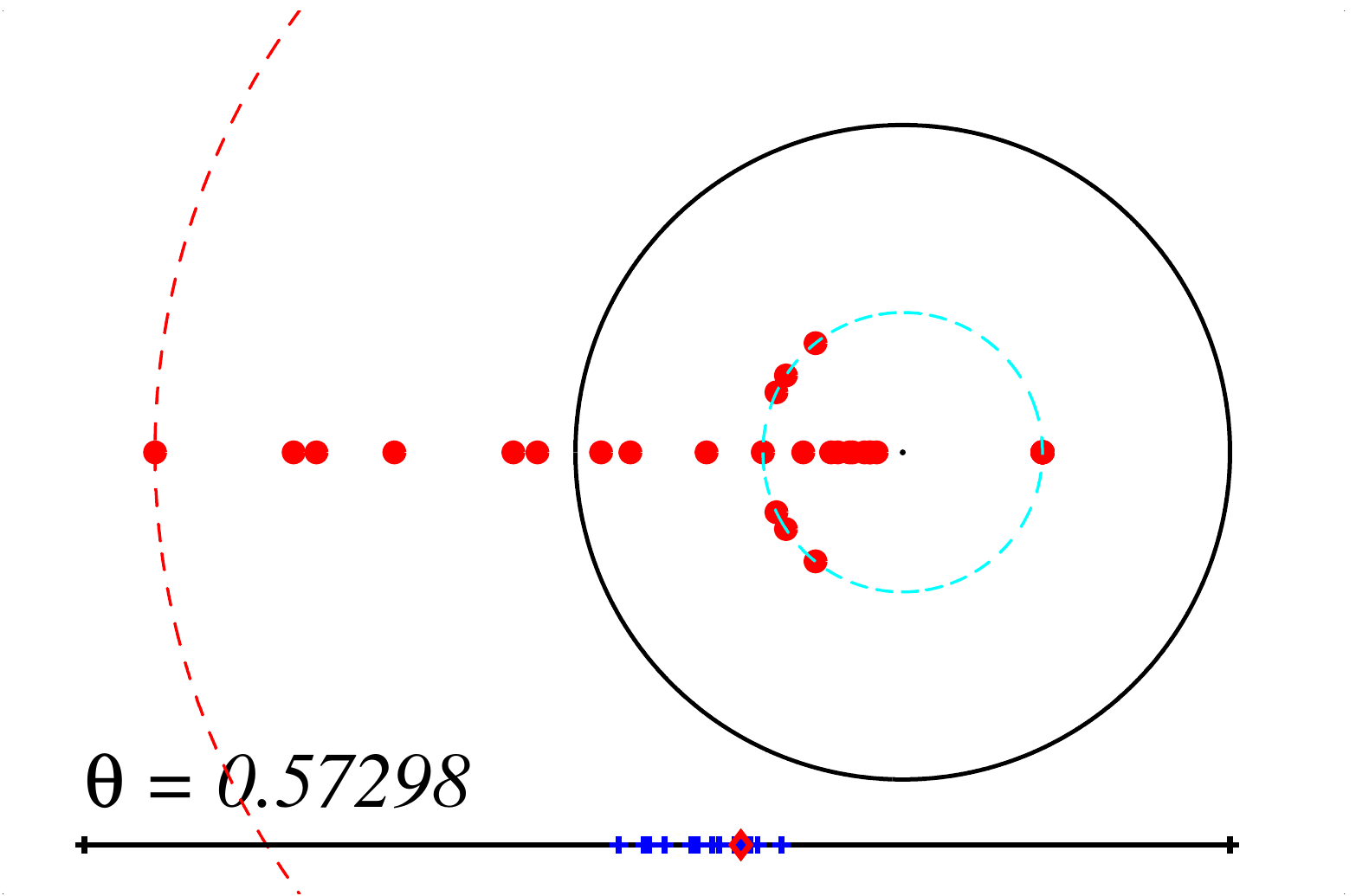}
\includegraphics[scale=0.248]{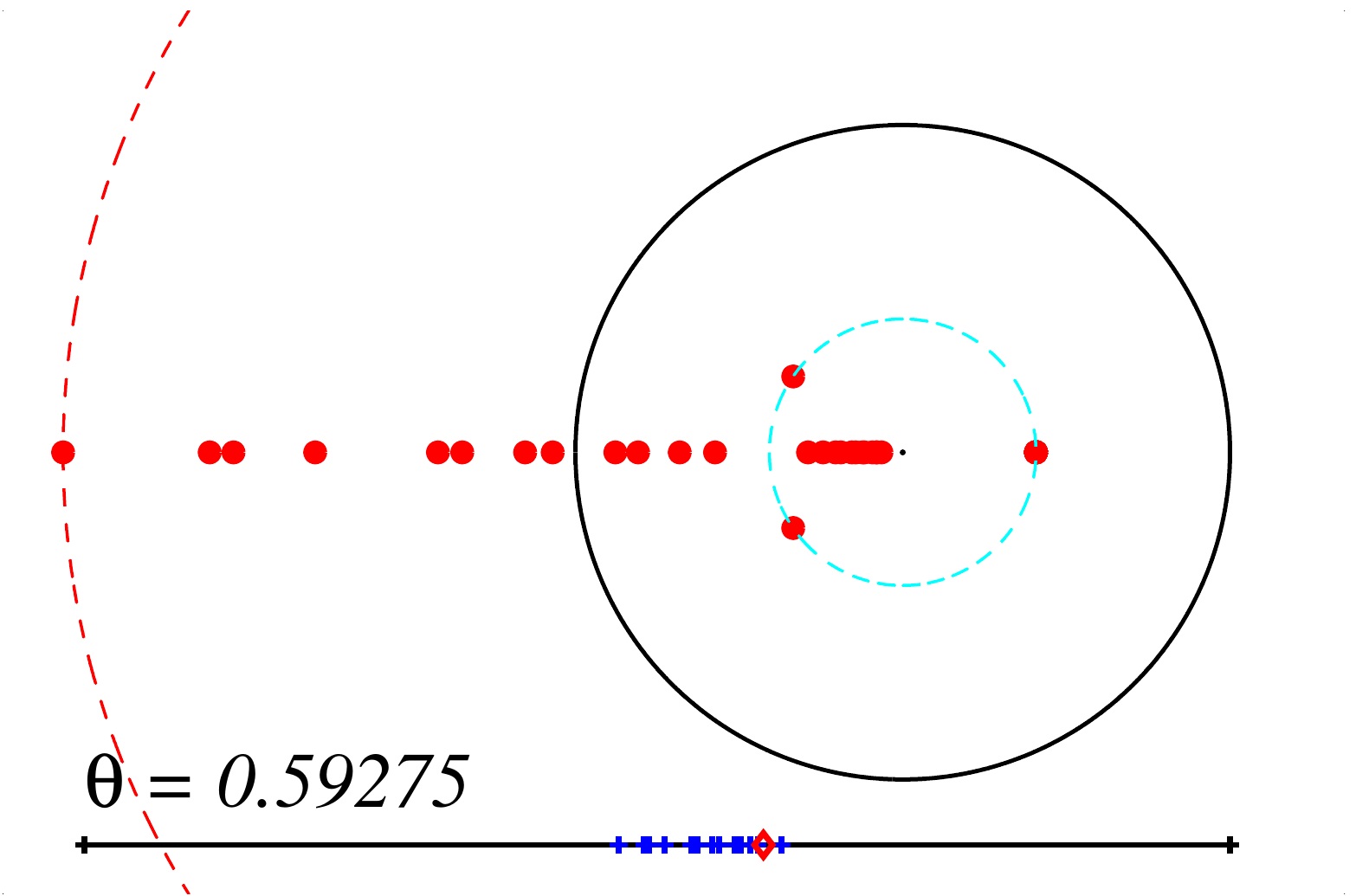}
\includegraphics[scale=0.248]{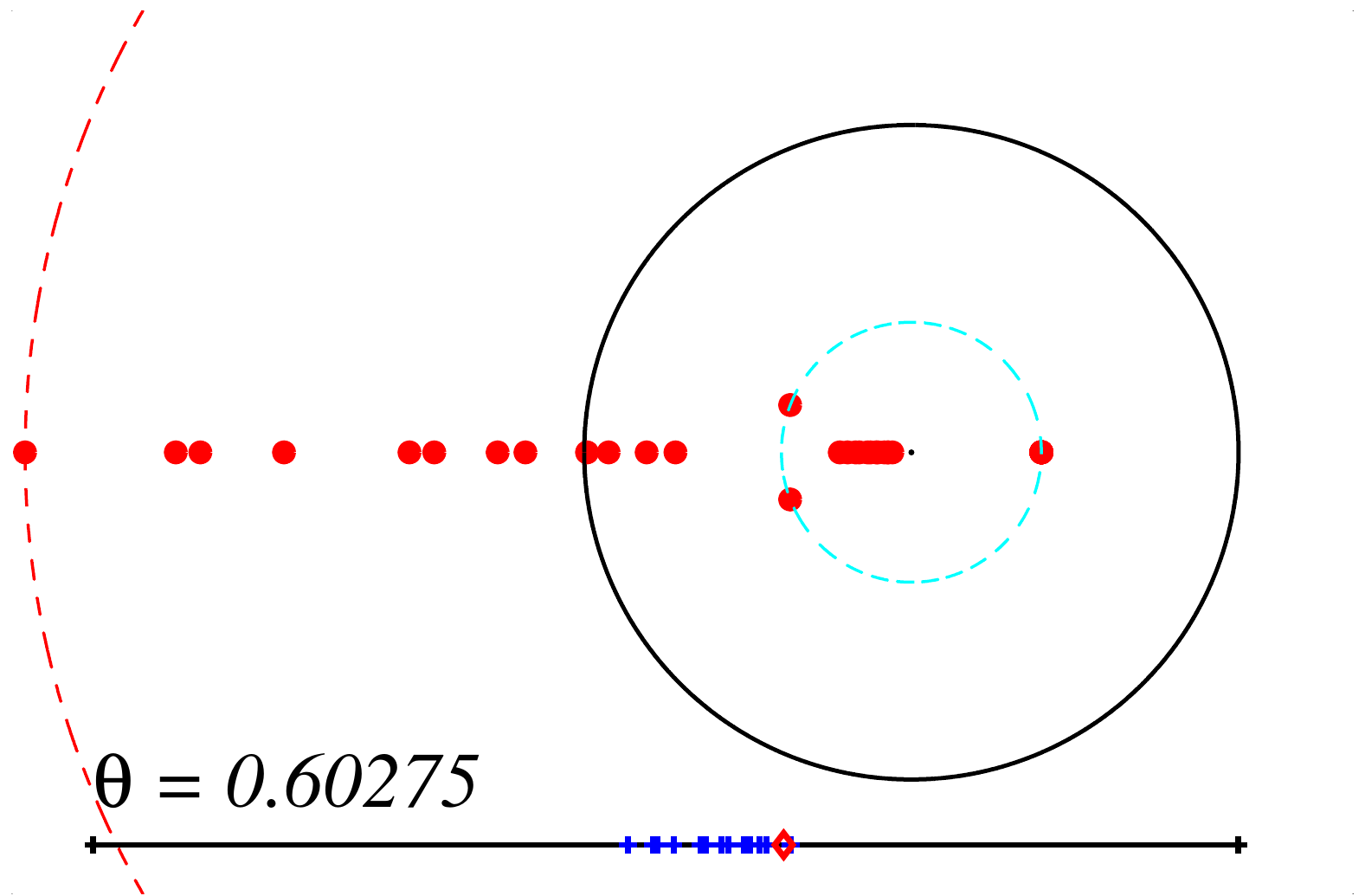}
\includegraphics[scale=0.248]{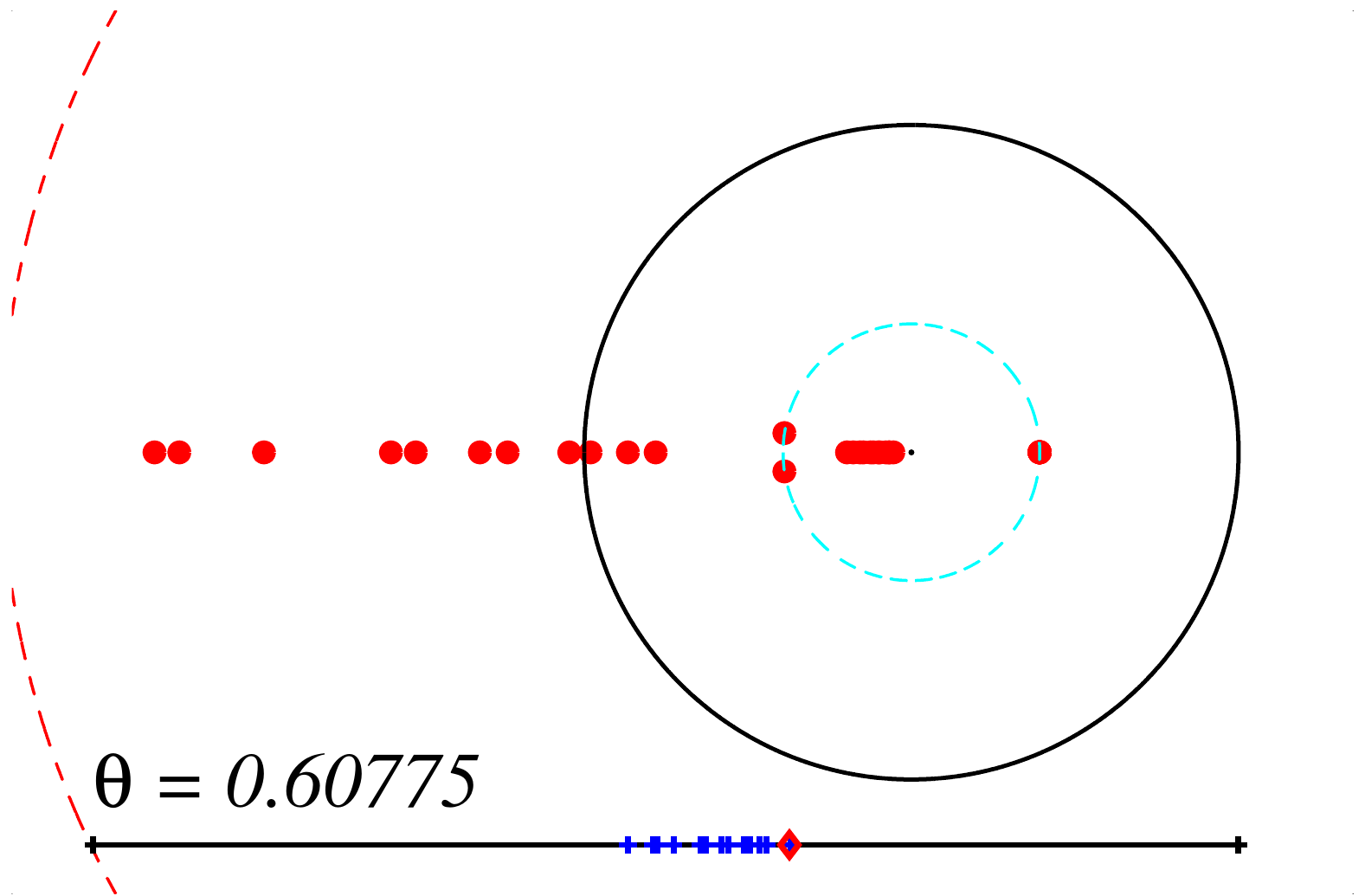}
\includegraphics[scale=0.248]{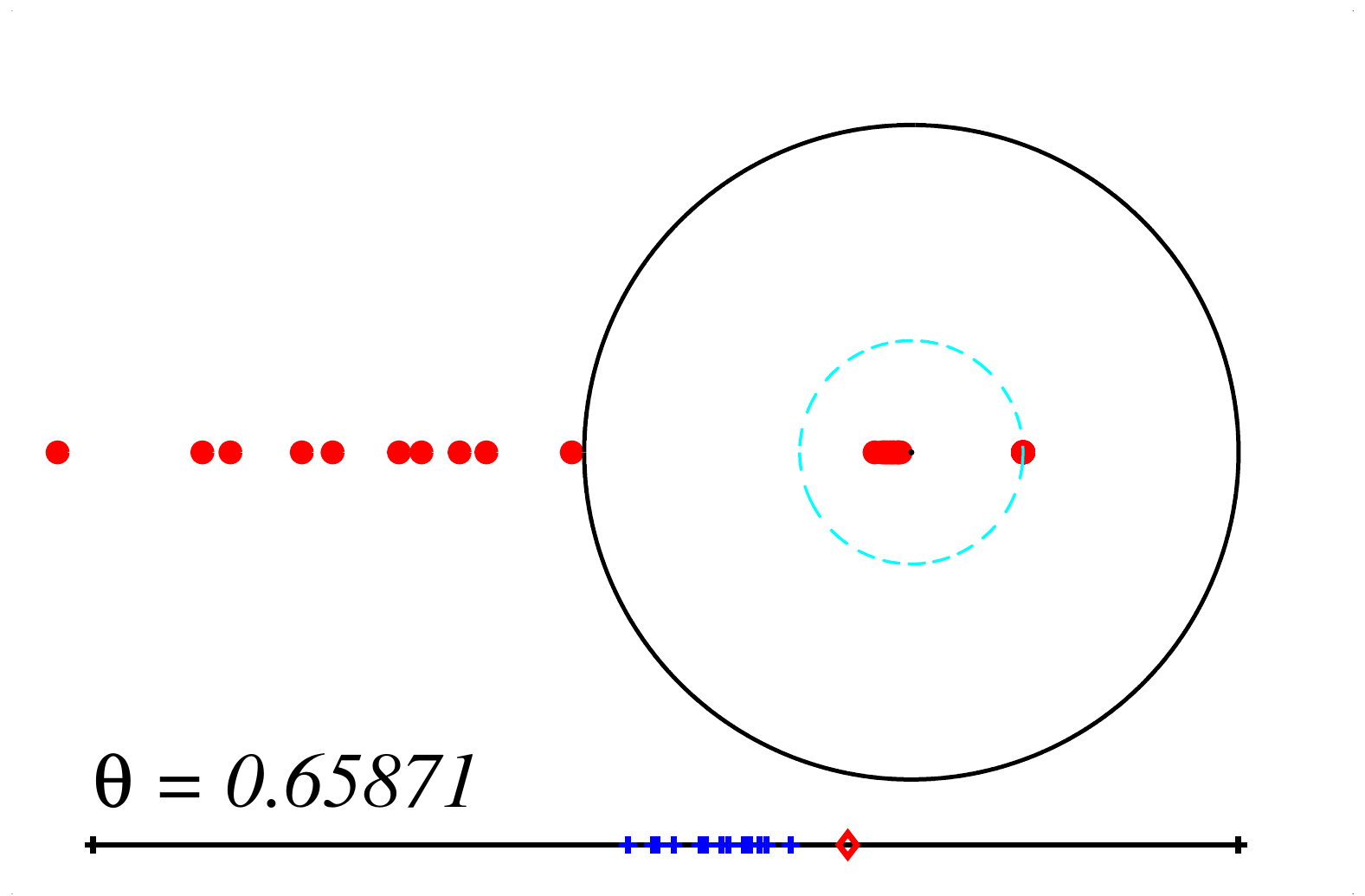}
\caption{The spectrum of $G_3(\theta)$ for many different values of $\theta$. The first 4 
pictures satisfying the condition $\theta^2\sigma_1^2-4(1-\theta)\lambda n<0$; in the 
fifth picture we have $\theta^2\sigma_1^2-4(1-\theta)\lambda n=0$ and the remaining ones 
represent the case where $\theta^2\sigma_1^2-4(1-\theta)\lambda n>0$.
\redcolor{The straight line represents (in a different scale) the interval $[0,1]$ on which are ploted 
some specific values of $\theta$ (blue marks), defined in \eqref{thetas}. The red diamond 
corresponds to the current value of $\theta$.}}
\label{fig_eigqtz1}
\end{figure}

In what follows, let us consider the functions $\delta_j:[0,1]\to\mathbb{R}$ defined by 
\begin{equation}
\label{deltas}
\delta_j(\theta)=\theta^2\sigma_j^2-4(1-\theta)\lambda n.
\end{equation}

The following straightforward result brings some basic properties of them, illustrated in 
Figure~\ref{fig_deltas}. 
\begin{lemma}
\label{lm_deltas}
Each function $\delta_j$, $j=1,\ldots,p$, is strictly increasing, from $-4\lambda n$ 
to $\sigma_j^2$ as $\theta$ goes from zero to $1$. Furthermore, these functions are 
sorted in decreasing order, $\delta_1\geq\delta_2\geq\cdots\geq\delta_p$, and their zeros, 
\begin{equation}
\label{thetas}
\bar\theta_j\stackrel{\rm def}{=}
\dfrac{-2\lambda n+2\textstyle\sqrt{\lambda n(\lambda n+\sigma_j^2)}}{\sigma_j^2},  
\end{equation}
are sorted in increasing order: 
$0<\bar\theta_1\leq\bar\theta_2\leq\cdots\leq\bar\theta_p<1$. 
\end{lemma}

\begin{figure}[htbp]
\centering
\includegraphics[scale=0.35]{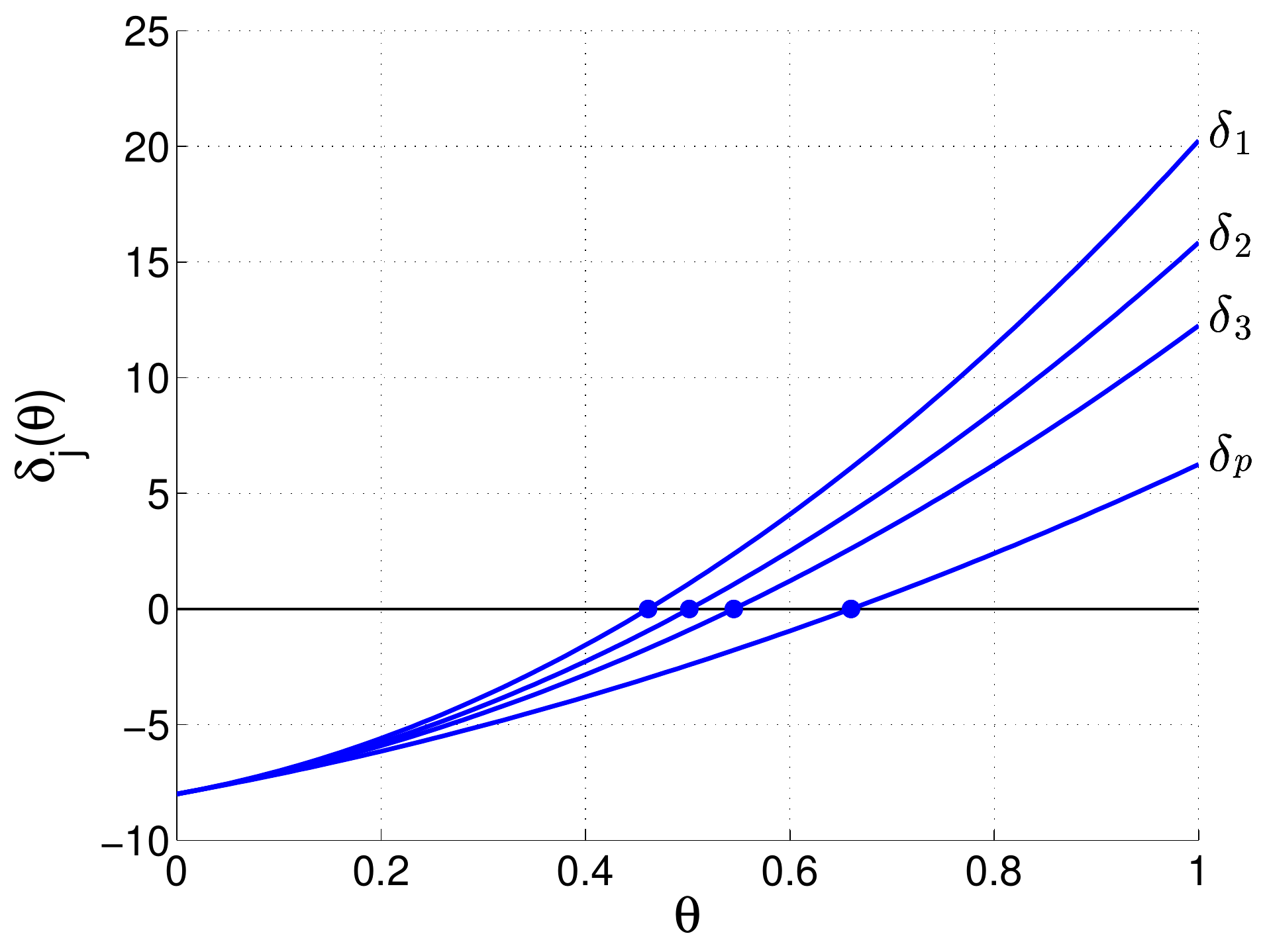}
\caption{The functions $\delta_j$, $j=1,\ldots,p$ and the properties stated in 
Lemma \ref{lm_deltas}.}
\label{fig_deltas}
\end{figure}

Now we shall study the spectrum of $G_3(\theta)$, given in Lemma \ref{lm_eigGqtz}. For this, 
let us denote 
\begin{equation}
\label{eig_G0}
\lambda_0(\theta)\stackrel{\rm def}{=}(1-\theta)
\end{equation}
and, for $j=1,\ldots,p$, 
\begin{equation}
\label{eig_G-}
\lambda_j^-(\theta)\stackrel{\rm def}{=}
\dfrac{1}{2\lambda n}\left(2(1-\theta)\lambda n-\theta^2\sigma_j^2-
\theta\sigma_j\textstyle\sqrt{\delta_j(\theta)}\right), 
\end{equation}
\begin{equation}
\label{eig_G+}
\lambda_j^+(\theta)\stackrel{\rm def}{=}
\dfrac{1}{2\lambda n}\left(2(1-\theta)\lambda n-\theta^2\sigma_j^2+
\theta\sigma_j\textstyle\sqrt{\delta_j(\theta)}\right)
\end{equation}
where $\delta_j$ is defined in \eqref{deltas}. 

\begin{lemma}
\label{eig_circle}
Consider $\bar\theta_1$ as defined in \eqref{thetas}. If $\theta\in[0,\bar\theta_1]$, then 
$$
|\lambda_j^+(\theta)|=|\lambda_j^-(\theta)|=1-\theta
$$
for all $j=1,\ldots,p$, which in turn implies that the spectral radius of $G_3(\theta)$ is 
$1-\theta$. 
\end{lemma}
\begin{proof}
Note that in this case we have $\delta_j(\theta)\leq 0$ for all $j=1,\ldots,p$. So, 
% $$
% \begin{array}{rcl}
% |\lambda_j^+(\theta)|^2 & = & |\lambda_j^-(\theta)|^2 \vspace{.2cm} \\
% & = & \dfrac{1}{4\lambda^2 n^2}
% \left(\Big(2(1-\theta)\lambda n-\theta^2\sigma_j^2\Big)^2-
% \theta^2\sigma_j^2\delta_j(\theta)\right) 
% \vspace{.2cm} \\
% & = & \dfrac{1}{4\lambda^2 n^2}
% \left(4(1-\theta)^2\lambda^2 n^2-4(1-\theta)\lambda n\theta^2\sigma_j^2+
% \theta^4\sigma_j^4-\theta^2\sigma_j^2\big(\theta^2\sigma_j^2-4(1-\theta)\lambda n\big)\right) 
% \vspace{.2cm} \\
% & = & (1-\theta)^2,
% \end{array}
% $$
$$
\begin{array}{rcl}
|\lambda_j^+(\theta)|^2 & = & |\lambda_j^-(\theta)|^2 \vspace{.2cm} \\
& = & \dfrac{1}{4\lambda^2 n^2}
\left(\Big(2(1-\theta)\lambda n-\theta^2\sigma_j^2\Big)^2-
\theta^2\sigma_j^2\delta_j(\theta)\right) 
\vspace{.2cm} \\
& = & (1-\theta)^2,
\end{array}
$$
yielding the desired result since $\theta\leq 1$.
\end{proof}

\medskip

It can be shown that the parameter $\theta=\theta_1^*$, defined in Theorem \ref{th_mfp1}, 
satisfies the conditions of Lemma \ref{eig_circle}. So, the spectral radius of 
$G_3(\theta_1^*)$ is $1-\theta_1^*$, exactly the same spectral radius of 
$G_1(\theta_1^*)=(1-\theta_1^*)I+\theta_1^* M_1$. This is shown in Figure~\ref{fig_eigqtz2}, 
together with the spectrum of $G_2(\theta_2^*)=(1-\theta_2^*)I+\theta_2^* M_2$. We also show 
in this figure (the right picture) the spectrum of $G_3(\theta_3^*)$, where $\theta_3^*$ is 
the optimal parameter. This parameter will be determined later, in Theorem \ref{conv_qtz}. 

\begin{figure}[htbp]
\centering
\includegraphics[scale=0.342]{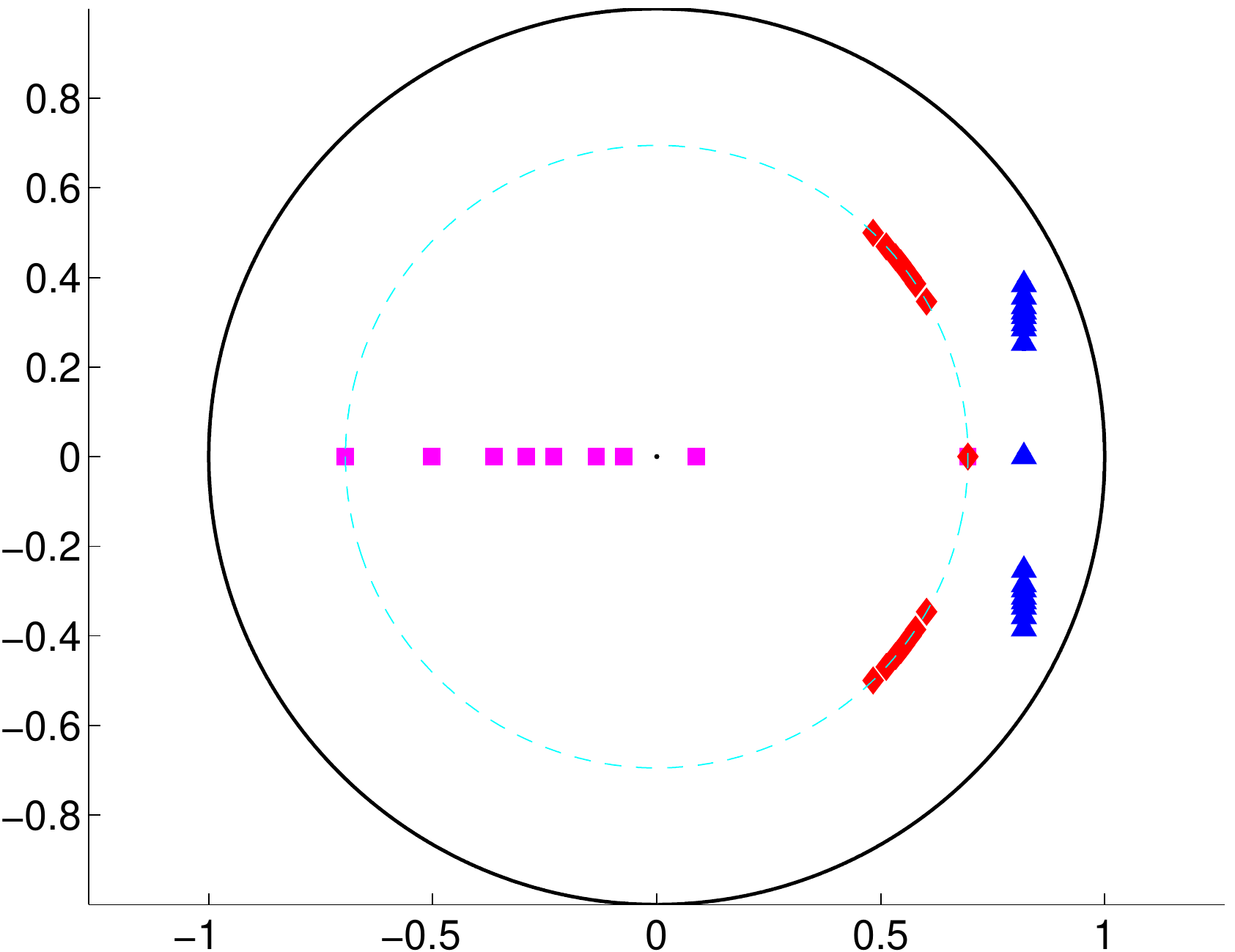}
\includegraphics[scale=0.342]{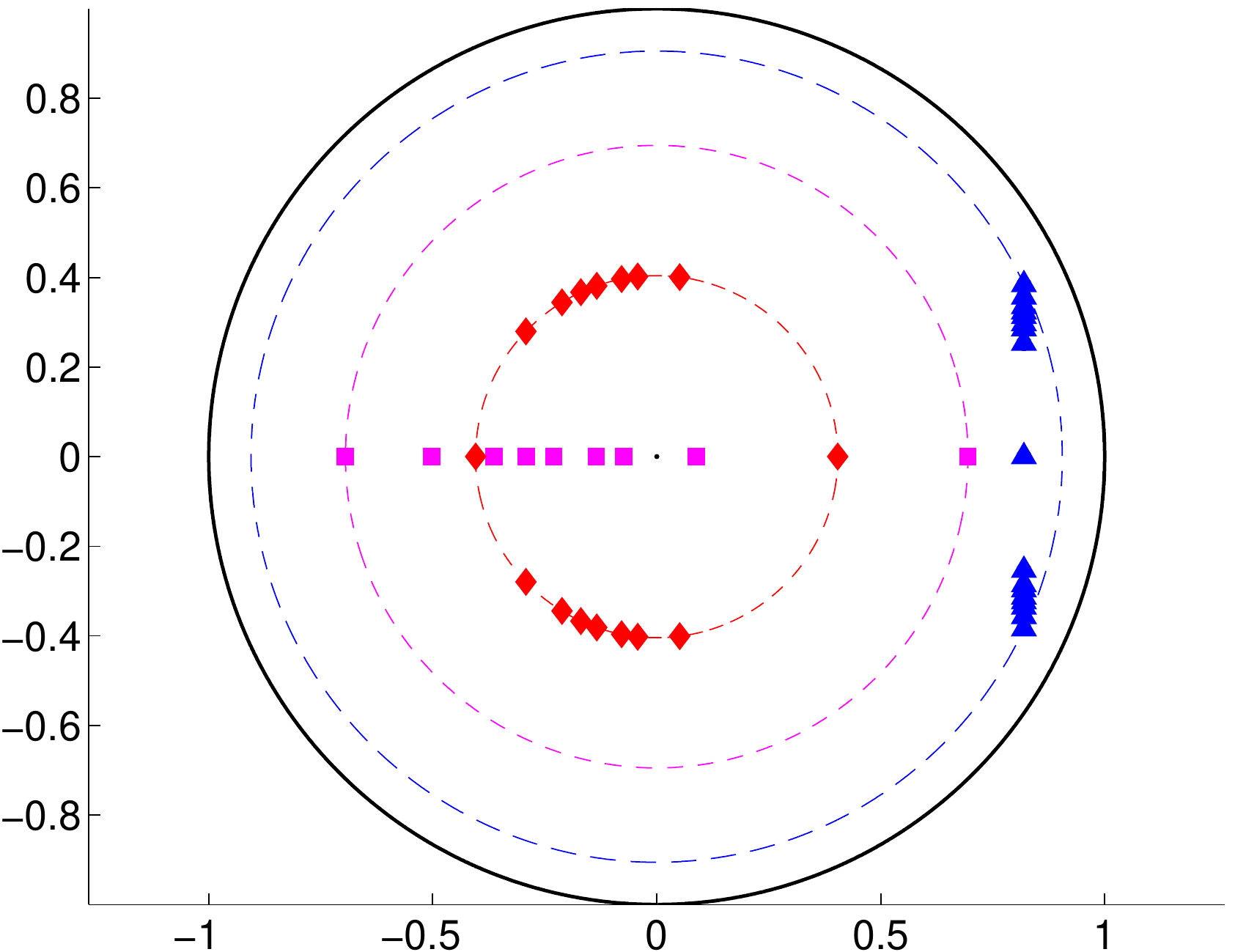}
\caption{On the left, the spectrum of $G_1(\theta_1^*)$ (magenta squares), $G_2(\theta_2^*)$ 
(blue triangles) and $G_3(\theta_1^*)$ (red diamonds). On the right, the spectrum of $G_3(\theta_3^*)$, where 
$\theta_3^*$ is the optimal parameter.}
\label{fig_eigqtz2}
\end{figure}

\begin{lemma}
\label{eig_real}
Consider $\bar\theta_j$, $j=1,\ldots,p$, as defined in \eqref{thetas}. 
If $\theta\in[\bar\theta_l,\bar\theta_{l+1}]$, then the eigenvalues 
$\lambda_j^+(\theta)$ and $\lambda_j^-(\theta)$, $j=1,\ldots,l$, are real 
numbers satisfying 
$$
\lambda_1^-(\theta)\leq\cdots\leq\lambda_l^-(\theta)\leq\theta-1
\leq\lambda_l^+(\theta)\leq\cdots\leq\lambda_1^+(\theta)\leq 0.
$$
On the other hand, for $j=l+1,\ldots,p$ we have 
$$
|\lambda_j^+(\theta)|=|\lambda_j^-(\theta)|=1-\theta
$$
Thus, the spectral radius of $G_3(\theta)$ is $-\lambda_1^-(\theta)$. 
\end{lemma}
\begin{proof}
We have $\delta_j(\theta)\geq 0$ for all $j=1,\ldots,l$. So, 
$$
\begin{array}{rcl}
\lambda_j^+(\theta)-(\theta-1) & = & 
\dfrac{1}{2\lambda n}\left(2(1-\theta)\lambda n-\theta^2\sigma_j^2+
\theta\sigma_j\sqrt{\delta_j(\theta)}\right)+1-\theta \vspace{.2cm} \\
& = & \dfrac{1}{2\lambda n}\left(4(1-\theta)\lambda n-\theta^2\sigma_j^2+
\theta\sigma_j\sqrt{\delta_j(\theta)}\right) \vspace{.2cm} \\
& = & \dfrac{1}{2\lambda n}\left(-\delta_j(\theta)+
\theta\sigma_j\sqrt{\delta_j(\theta)}\right) \vspace{.2cm} \\
& = & \dfrac{\sqrt{\delta_j(\theta)}}{2\lambda n}
\left(\theta\sigma_j-\sqrt{\delta_j(\theta)}\right)\geq 0.
\end{array}
$$
Furthermore, 
$$
\Big(\theta\sigma_j\textstyle\sqrt{\delta_j(\theta)}\Big)^2 = 
\theta^2\sigma_j^2\Big(\theta^2\sigma_j^2-4(1-\theta)\lambda n\Big) 
\leq \Big(\theta^2\sigma_j^2-2(1-\theta)\lambda n\Big)^2.
$$
Since $\theta^2\sigma_j^2-2(1-\theta)\lambda n=\delta_j(\theta)+2(1-\theta)\lambda n\geq 0$, 
$$
\lambda_j^+(\theta)=\dfrac{1}{2\lambda n}\left(2(1-\theta)\lambda n-\theta^2\sigma_j^2+
\theta\sigma_j\textstyle\sqrt{\delta_j(\theta)}\right)\leq 0.
$$
Now, note that 
$$
\begin{array}{rcl}
\lambda_j^-(\theta)-(\theta-1) & = & 
\dfrac{1}{2\lambda n}\left(2(1-\theta)\lambda n-\theta^2\sigma_j^2-
\theta\sigma_j\sqrt{\delta_j(\theta)}\right)+1-\theta \vspace{.2cm} \\
& = & \dfrac{1}{2\lambda n}\Big(-\delta_j(\theta)-
\theta\sigma_j\sqrt{\delta_j(\theta)}\Big) \leq 0.
\end{array}
$$
Moreover, from Lemma \ref{lm_deltas} and the definition of $\sigma_j$, we have 
$\delta_1(\theta)\geq\cdots\geq\delta_l(\theta)$ and 
$\sigma_1\geq\cdots\geq\sigma_l$, which imply that 
$\lambda_1^-(\theta)\leq\cdots\leq\lambda_l^-(\theta)$. 
The inequality $\lambda_l^+(\theta)\leq\cdots\leq\lambda_1^+(\theta)$ 
follows from the fact that the function $$[\sqrt{a},\infty)\ni s\mapsto -s^2+s\sqrt{s^2-a}$$ 
is increasing. 
Finally, for $j=l+1,\ldots,p$ we have $\delta_j(\theta)\leq 0$ and, by the same 
argument used in Lemma \ref{eig_circle}, we conclude that 
$
|\lambda_j^+(\theta)|=|\lambda_j^-(\theta)|=1-\theta.
$
\end{proof}

\medskip

From Lemmas \ref{eig_circle} and \ref{eig_real} we can conclude that $\bar\theta_1$ is 
the threshold value for $\theta$ after which the eigenvalues of $G_3(\theta)$ start departing 
the circle of radius $1-\theta$. The next result presents the threshold after which the 
eigenvalues are all real.

\begin{lemma}
\label{eig_real2}
Consider $\bar\theta_p$ as defined in \eqref{thetas}. 
If $\theta\geq\bar\theta_p$, then the eigenvalues 
$\lambda_j^+(\theta)$ and $\lambda_j^-(\theta)$, $j=1,\ldots,p$, are real 
numbers satisfying 
$$
\lambda_1^-(\theta)\leq\cdots\leq\lambda_p^-(\theta)\leq\theta-1
\leq\lambda_p^+(\theta)\leq\cdots\leq\lambda_1^+(\theta)\leq 0.
$$
Thus, the spectral radius of $G_3(\theta)$ is $-\lambda_1^-(\theta)$. 
\end{lemma}
\begin{proof}
The same presented for Lemma \ref{eig_real}. 
\end{proof}

\medskip

Using the previous results, we can finally establish the convergence of 
Algorithm \ref{alg_qtz} (Deterministic Quartz). 

\begin{theorem}
\label{conv_qtz}
Let $w^0\in\mathbb{R}^d$ and $\alpha^0\in\mathbb{R}^N$ be arbitrary and consider the sequence 
$(w^k,\alpha^k)_{k\in\mathbb{N}}$ generated by Algorithm \ref{alg_qtz} with 
$\theta\in\left(0,\dfrac{2\sqrt{\lambda n}}{\sqrt{\lambda n}+\sigma_1}\right)$. 
Then the sequence $(w^k,\alpha^k)$ converges to the (unique) solution of the 
problem \eqref{pd_prob} at an asymptotic linear rate of 
$$
\rho_3(\theta)\stackrel{\rm def}{=}\left\{\begin{array}{l}
1-\theta , \, \mbox{if } \theta\in(0,\bar\theta_1] \vspace{.2cm} \\
\dfrac{1}{2\lambda n}\left(\theta\sigma_1\sqrt{\delta_1(\theta)}+
\theta^2\sigma_1^2-2(1-\theta)\lambda n\right),\, \mbox{if } \theta\geq\bar\theta_1,
\end{array}\right.
$$
where 
$
\bar\theta_1=
\dfrac{-2\lambda n+2\displaystyle\sqrt{\lambda n(\lambda n+\sigma_1^2)}}{\sigma_1^2}.
$
Furthermore, if we choose $\theta_3^*\stackrel{\rm def}{=}\bar\theta_1$, 
then the (theoretical) convergence rate is optimal and it is equal to 
$
\rho_3^*\stackrel{\rm def}{=}1-\theta_3^*.
$ 
\end{theorem}
\begin{proof}
Since Algorithm~\ref{alg_qtz} can be represented by \eqref{qtzfixp}, we need to show 
that $\rho(G_3(\theta))$, the spectral radius of $G_3(\theta)$, is less than $1$. 
First, note that by Lemmas \ref{eig_circle}, \ref{eig_real} and \ref{eig_real2}, we have 
$\rho(G_3(\theta))=\rho_3(\theta)$. 
Using Lemma \ref{lm_deltas} we conclude that the function 
$\theta\mapsto\rho_3(\theta)$ is increasing on the interval $[\bar\theta_1,\infty)$, which 
means that its minimum is attained at $\bar\theta_1$.
To finish the proof, it is enough to prove that $\rho_3(\theta)=1$ if and only if 
$$
\theta=\dfrac{2\sqrt{\lambda n}}{\sqrt{\lambda n}+\sigma_1}.
$$
Note that 
$$
\begin{array}{rcl}
\rho_3(\theta)=1 & \Leftrightarrow & 
\theta\sigma_1\sqrt{\delta_1(\theta)}+\theta^2\sigma_1^2-2(1-\theta)\lambda n=2\lambda n
\vspace{.2cm} \\
& \Rightarrow & 
\theta^2\sigma_1^2\delta_1(\theta)=\Big(2(2-\theta)\lambda n-\theta^2\sigma_1^2\Big)^2
\vspace{.2cm} \\
& \Leftrightarrow & 
\dfrac{2-\theta}{\theta}=\dfrac{\sigma_1}{\sqrt{\lambda n}}
\vspace{.2cm} \\
& \Leftrightarrow & 
\theta=\dfrac{2\sqrt{\lambda n}}{\sqrt{\lambda n}+\sigma_1}
\end{array}
$$
and 
$$
\begin{array}{rcl}
\theta=\dfrac{2\sqrt{\lambda n}}{\sqrt{\lambda n}+\sigma_1} & \Leftrightarrow & 
\theta^2\sigma_1^2=\lambda n(2-\theta)^2
\vspace{.2cm} \\
& \Leftrightarrow & 
\theta\sigma_1\sqrt{\delta_1(\theta)}+\theta^2\sigma_1^2-2(1-\theta)\lambda n=2\lambda n,
\end{array}
$$
completing the proof. 
\end{proof}

\medskip

It is worth noting that if the spectral radius of $M_1$ is less than $1$, that is, if 
$\sigma_1^2<\lambda n$, then Algorithms \ref{alg_mfp1}, \ref{alg_mfp2} and 
\ref{alg_qtz} converge for any choice of $\theta\in(0,1]$. Indeed, 
in this case we have
$$
\dfrac{2\lambda n}{\lambda n+\sigma_1^2}>
\dfrac{2\sqrt{\lambda n}}{\sqrt{\lambda n}+\sigma_1}>1,
$$
which implies that the set of admissible values for $\theta$ established in 
Theorems \ref{th_mfp1}, \ref{th_mfp2} and \ref{conv_qtz} contains the whole 
interval $(0,1]$. 

On the other hand, if $\sigma_1^2\geq\lambda n$, the convergence of these algorithms 
is more restrictive. Moreover, in this case we have 
$$
\dfrac{2\lambda n}{\lambda n+\sigma_1^2}\leq
\dfrac{2\sqrt{\lambda n}}{\sqrt{\lambda n}+\sigma_1}\leq 1,
$$
which means that Algorithm \ref{alg_qtz} has a broader range for $\theta$ than 
Algorithms \ref{alg_mfp1} and \ref{alg_mfp2}.

\subsection{Complexity results}
\label{sec_cpxqtz}
Taking into account \eqref{cpx0}, \eqref{cond_f}, the relation $\log(1-\theta)\approx-\theta$ 
and Theorem \ref{conv_qtz}, we 
conclude that the complexity of our Accelerated Fixed Point Method, 
Algorithm \ref{alg_qtz}, is proportional to 
\begin{equation}
\label{cpxqtz1}
\dfrac{1}{2\theta_3^*}=
\dfrac{\sigma_1^2}{-4\lambda n+4\displaystyle\sqrt{\lambda n(\lambda n+\sigma_1^2)}} \overset{\eqref{cond_f}}{=}
\dfrac{\kappa - \lambda n }{4(\displaystyle\sqrt{\lambda n\kappa}-\lambda n)}.
\end{equation}

Note that in the case when $\lambda=1/n$, as is typical in machine learning applications, we can write
\begin{equation}
\label{cpxqtz}
\dfrac{1}{2\theta_3^*}\overset{\eqref{cpxqtz1}}{=}\frac{\kappa -1}{4(\sqrt{\kappa}-1)}=\frac{\sqrt{\kappa}+1}{4}.
\end{equation}

{\bf This is very surprising as it means that we are achieving the optimal accelerated Nesterov 
rate $\tilde{O}(\sqrt{\kappa})$.}

\section{Extensions}
\label{sec_qtze}
In this section we discuss some variants of Algorithm \ref{alg_qtz}. The first one 
 consists of switching the order of the computations, 
updating the dual variable first and then the primal one. 

The second approach updates the primal 
variable enforcing the first relation of the optimality conditions given by \eqref{optM2w} and 
using the relaxation parameter $\theta$ only to update the dual variable. 

\subsection{Switching the update order}
This approach updates the dual variable $\alpha$ first and then updates the primal 
variable $w$ using the new information about $\alpha$. This is summarized in the following scheme. 
\begin{equation}
\label{qtznew}
\left\{\begin{array}{l}
\alpha^{k+1}=(1-\theta)\alpha^k+\theta(y-A^Tw^k) \vspace{.15cm} \\ 
w^{k+1}=(1-\theta)w^k+\theta \dfrac{1}{\lambda n}A\alpha^{k+1}.
\end{array}
\right.
\end{equation}

As we shall see now, this scheme provides the same complexity results as Algorithm \ref{alg_qtz}. 
To see this, note that the iteration \eqref{qtznew} is equivalent to
$$
\left(\begin{array}{cc} I & -\frac{\theta}{\lambda n}A \\ 0 & I \end{array}\right)
\left(\begin{array}{c} w^{k+1} \\ \alpha^{k+1}\end{array}\right)=
\left(\begin{array}{cc} (1-\theta)I & 0  \vspace{.15cm}\\ 
-\theta A^T & (1-\theta)I \end{array}\right)
\left(\begin{array}{c} w^k \\ \alpha^k\end{array}\right)+
\left(\begin{array}{c} 0 \\ \theta y \end{array}\right)
$$
or in a compact way, 
$$
x^{k+1}=G(\theta)x^k+f
$$
with 
% $$
% G(\theta)=\left(\begin{array}{cc} I & -\frac{\theta}{\lambda n}A \\ 0 & I \end{array}\right)^{-1}
% \left(\begin{array}{cc} (1-\theta)I & 0  \vspace{.15cm}\\ 
% -\theta A^T & (1-\theta)I \end{array}\right)=(1-\theta)I+\theta\left(
% \begin{array}{cc} -\frac{\theta}{\lambda n}AA^T & \frac{1-\theta}{\lambda n}A \vspace{.15cm}\\ 
% -A^T & 0 \end{array}\right).
% $$
$$
G(\theta)=(1-\theta)I+\theta\left(
\begin{array}{cc} -\frac{\theta}{\lambda n}AA^T & \frac{1-\theta}{\lambda n}A \vspace{.15cm}\\ 
-A^T & 0 \end{array}\right).
$$

It can be shown that the matrix $G(\theta)$ has exactly the same spectrum of $G_3(\theta)$, 
defined in \eqref{Gqtz}. So, the convergence result is also the same, which we state again 
for convenience.

\begin{theorem}
\label{conv_qtznew}
Let $w^0\in\mathbb{R}^d$ and $\alpha^0\in\mathbb{R}^N$ be arbitrary and consider the sequence 
$(w^k,\alpha^k)_{k\in\mathbb{N}}$ defined by \eqref{qtznew} with 
$\theta\in\left(0,\dfrac{2\sqrt{\lambda n}}{\sqrt{\lambda n}+\sigma_1}\right)$. 
Then the sequence $(w^k,\alpha^k)$ converges to the (unique) solution of the 
problem \eqref{pd_prob} at an asymptotic linear rate of 
$$
\rho_3(\theta)=\left\{\begin{array}{l}
1-\theta , \, \mbox{if } \theta\in(0,\bar\theta_1] \vspace{.2cm} \\
\dfrac{1}{2\lambda n}\left(\theta\sigma_1\sqrt{\delta_1(\theta)}+
\theta^2\sigma_1^2-2(1-\theta)\lambda n\right),\, \mbox{if } \theta\geq\bar\theta_1,
\end{array}\right.
$$
where 
$
\bar\theta_1=
\dfrac{-2\lambda n+2\displaystyle\sqrt{\lambda n(\lambda n+\sigma_1^2)}}{\sigma_1^2}.
$
Furthermore, if we choose $\theta_3^*=\bar\theta_1$, 
then the (theoretical) convergence rate is optimal and it is equal to 
$
\rho_3^*=1-\theta_3^*.
$ 
\end{theorem}

\subsection{Maintaining primal-dual relationship}
\label{secqtzmod}
The second approach updates the primal variable enforcing the first relation of the 
optimality conditions given by \eqref{optM2w} and uses the relaxation parameter 
$\theta$ only to update the dual variable, as described in the following scheme. 
\begin{equation}
\label{qtzmod}
\left\{\begin{array}{l}
w^{k+1}=\dfrac{1}{\lambda n}A\alpha^k  \vspace{.15cm}\\ 
\alpha^{k+1}=(1-\theta)\alpha^k+\theta(y-A^Tw^{k+1}).
\end{array}
\right.
\end{equation}
Differently from the previous case, this scheme cannot achieve accelerated convergence. 
Indeed, note first that the scheme \eqref{qtzmod} can be written as 
$$
\left(\begin{array}{cc} I & 0 \\ \theta A^T & I \end{array}\right)
\left(\begin{array}{c} w^{k+1} \\ \alpha^{k+1}\end{array}\right)=
\left(\begin{array}{cc} 0 & \frac{1}{\lambda n}A  \vspace{.15cm}\\ 
0 & (1-\theta)I \end{array}\right)
\left(\begin{array}{c} w^k \\ \alpha^k\end{array}\right)+
\left(\begin{array}{c} 0 \\ \theta y \end{array}\right)
$$
or in a compact way, 
$$
x^{k+1}=G(\theta)x^k+f
$$
with 
% $$
% G(\theta)=\left(\begin{array}{cc} I & 0 \\ \theta A^T & I \end{array}\right)^{-1}
% \left(\begin{array}{cc} 0 & \frac{1}{\lambda n}A  \vspace{.15cm}\\ 
% 0 & (1-\theta)I \end{array}\right)=
% \left(\begin{array}{cc} 0 & \frac{1}{\lambda n}A \vspace{.15cm}\\ 
% 0 & (1-\theta)I-\frac{\theta}{\lambda n}A^TA \end{array}\right).
% $$
$$
G(\theta)=
\left(\begin{array}{cc} 0 & \frac{1}{\lambda n}A \vspace{.15cm}\\ 
0 & (1-\theta)I-\frac{\theta}{\lambda n}A^TA \end{array}\right).
$$
We can conclude that the eigenvalues of this matrix are  
$$
\left\{1-\theta-\dfrac{\theta\sigma_j^2}{\lambda n}\;,\; j=1,\ldots,p
\right\}\cup\{1-\theta\},
$$
exactly the same of the matrix $G_1(\theta)$, the iteration matrix of Algorithm \ref{alg_mfp} 
with employment of $M_1$. So, the complexity analysis here is the same as that one established in 
Theorem \ref{th_mfp1}. 

\subsection{Maintaining primal-dual relationship 2}
For the sake of completeness, we present next the method where we keep the second relationship 
intact and include $\theta$ in the first relationship. This leads to
\begin{equation}
\label{qtzmod2}
\left\{\begin{array}{l}
\alpha^{k+1}=y-A^Tw^{k}
 \vspace{.15cm}\\ 
w^{k+1}=(1-\theta)w^k+\dfrac{\theta}{\lambda n}A\alpha^{k+1}.
\end{array}
\right.
\end{equation}
Here we obtain the same convergence results as the ones described in Section \ref{secqtzmod}.
In fact, the relations above can be written as 
$$
\left(\begin{array}{cc} 0 & I \\ I & -\frac{\theta}{\lambda n}A \end{array}\right)
\left(\begin{array}{c} w^{k+1} \\ \alpha^{k+1}\end{array}\right)=
\left(\begin{array}{cc} -A^T & 0  \vspace{.15cm}\\ 
(1-\theta)I & 0 \end{array}\right)
\left(\begin{array}{c} w^k \\ \alpha^k\end{array}\right)+
\left(\begin{array}{c} y \\ 0 \end{array}\right)
$$
or in a compact way, 
$
x^{k+1}=G(\theta)x^k+f
$
with 
% $$
% G(\theta)=\left(\begin{array}{cc} 0 & I \\ I & -\frac{\theta}{\lambda n}A \end{array}\right)^{-1}
% \left(\begin{array}{cc} -A^T & 0  \vspace{.15cm}\\ 
% (1-\theta)I & 0 \end{array}\right)=
% \left(\begin{array}{cc} (1-\theta)I-\frac{\theta}{\lambda n}AA^T & 0 \vspace{.15cm}\\ 
% -A^T & 0 \end{array}\right).
% $$
$$
G(\theta)=
\left(\begin{array}{cc} (1-\theta)I-\frac{\theta}{\lambda n}AA^T & 0 \vspace{.15cm}\\ 
-A^T & 0 \end{array}\right).
$$
We can conclude that the eigenvalues of this matrix are  
$$
\left\{1-\theta-\dfrac{\theta\sigma_j^2}{\lambda n}\;,\; j=1,\ldots,p
\right\}\cup\{1-\theta\},
$$
exactly the same of the matrix $G_1(\theta)$, the iteration matrix of Algorithm \ref{alg_mfp} 
with employment of $M_1$. So, the complexity analysis here is the same as that one established in 
Theorem \ref{th_mfp1}. 

Observe that in \eqref{qtzmod} we have 
$$
w^{k+1}=\phi_1(\alpha^k)\quad\mbox{and}\quad\alpha^{k+1}=\phi_2(\theta,\alpha^k,w^{k+1}).
$$
On the other hand, in \eqref{qtzmod2} we have 
$$
\alpha^{k+1}=\phi_3(w^k)\quad\mbox{and}\quad w^{k+1}=\phi_4(\theta,w^k,\alpha^{k+1}).
$$ 
It is worth noting that if we update the variables as 
$$
\alpha^{k+1}=\phi_2(\theta,\alpha^k,w^k)\quad\mbox{and}\quad w^{k+1}=\phi_1(\alpha^{k+1})
$$
or
$$
w^{k+1}=\phi_4(\theta,w^k,\alpha^{k})\quad\mbox{and}\quad\alpha^{k+1}=\phi_3(w^{k+1})
$$ 
we obtain 
$$
\left(\begin{array}{cc} -\frac{\theta}{\lambda n}AA^T & \frac{(1-\theta)}{\lambda n}A \vspace{.15cm}\\ 
-\theta A^T & (1-\theta)I \end{array}\right)
\quad\mbox{and}\quad
\left(\begin{array}{cc} (1-\theta)I & \frac{\theta}{\lambda n}A \vspace{.15cm}\\ 
- (1-\theta)A^T & -\frac{\theta}{\lambda n}A^TA \end{array}\right)
$$
as the associated iteration matrices, respectively. Moreover, we can conclude that they also 
have the same spectrum of $G_1(\theta)$. So, the complexity analysis is the same as that one established 
in Theorem \ref{th_mfp1}.

\section{Numerical Experiments}
\label{sec_num} 
In this section we present a comparison among the methods discussed in this work. Besides 
a table with the convergence rates and complexity bounds, we show here some numerical tests 
performed to illustrate the properties of Algorithms \ref{alg_mfp1}, \ref{alg_mfp1} and \ref{alg_qtz} as well 
as of the extensions \eqref{qtznew} and \eqref{qtzmod} applied to solve the primal-dual ridge 
regression problem stated in \eqref{pd_prob}. We refer to Algorithm~\ref{alg_qtz} as Quartz and 
the extensions \eqref{qtznew} and \eqref{qtzmod} as New Quartz and Modified Quartz, respectively.
The name Quartz is due to the fact that Algorithm \ref{alg_qtz} is a deterministic version of a 
randomized primal-dual algorithm proposed and analyzed by Qu, Richt\'arik and Zhang \cite{Quartz}.  

We summarize the main features of these methods in Table \ref{table2} which brings the range of 
the parameter to ensure convergence, the optimal convergence rates, the complexity and the cost 
per iteration of each method.  For instance, the two versions of Algorithm \ref{alg_mfp} have the 
same range for theta. The usage of $M_1$ provides best convergence rate compared with using $M_2$. 
However, it requires more calculations per iteration: the major computational tasks to be 
performed are computation of the matrix-vector products $AA^Tw$ and $A^TA\alpha$, while the 
use of $M_2$ needs the computation of $A\alpha$ and $A^Tw$. 

Surprisingly, Algorithm \ref{alg_qtz} has shown to be the best from both the theoretical point 
of view and the numerical experiments and with the same cost as the computation 
of $A\alpha$ and $A^Tw$. 

We also point out that the modified Quartz, \eqref{qtzmod}, did not have here the same 
performance as the randomized version studied in \cite{Quartz}. 

\begin{table}[htbp]
\begin{center}
\renewcommand{\arraystretch}{2.3}
{\small
\begin{tabular}{|c|c|c|c|c|}
\hline
& Range of $\theta$ & Optimal rate & Complexity & Cost/iteration \\
\cline{1-5}
PDFP1$(\theta)$ & $\left(0,\dfrac{2\lambda n}{\lambda n+\sigma_1^2}\right)$  & 
$\dfrac{\sigma_1^2}{2\lambda n+\sigma_1^2}$ & \eqref{cpx1} & $10dN+5d+9N$\\
\cline{1-5}
PDFP2$(\theta)$ & $\left(0,\dfrac{2\lambda n}{\lambda n+\sigma_1^2}\right)$ & 
$\sqrt{\dfrac{\sigma_1^2}{\lambda n+\sigma_1^2}}$ & \eqref{cpx2}  & $6dN+5d+9N$ \\
\cline{1-5}
QTZ$(\theta)$ &  $\left(0,\dfrac{2\sqrt{\lambda n}}{\sqrt{\lambda n}+\sigma_1}\right)$ & 
$1-\theta_3^*$ & \eqref{cpxqtz} & $6dN+5d+9N$ \\
\cline{1-5}
NQTZ$(\theta)$ &  $\left(0,\dfrac{2\sqrt{\lambda n}}{\sqrt{\lambda n}+\sigma_1}\right)$ & 
$1-\theta_3^*$ & \eqref{cpxqtz} & $6dN+5d+9N$ \\
\cline{1-5}
MQTZ$(\theta)$ & $\left(0,\dfrac{2\lambda n}{\lambda n+\sigma_1^2}\right)$ & 
$\dfrac{\sigma_1^2}{2\lambda n+\sigma_1^2}$ &\eqref{cpx1} & $6dN+3d+9N$ \\
\hline
\end{tabular}
}
\end{center}
\caption{\redcolor{Comparison between the ranges of $\theta$ to ensure convergence, optimal 
convergence rates, complexity and cost per iteration (\# of arithmetic operations) of the 
algorithms proposed in this paper: Algorithm~\ref{alg_mfp1}, indicated by PDFP1($\theta$); 
Algorithm \ref{alg_mfp2}, denoted by PDFP2($\theta$); Algorithm~\ref{alg_qtz}, 
QTZ($\theta$) and the extensions \eqref{qtznew} (New Quartz) and \eqref{qtzmod} 
(Modified Quartz), indicated by NQTZ($\theta$) and MQTZ($\theta$), respectively.}}
\label{table2}
\end{table}

Figure \ref{fig_fxpqtz} illustrates these features, showing the 
primal-dual objective values against the number of iterations. The dimensions considered were 
$d=10$, $m=1$ and $n=500$. \redcolor{We adopted the optimal parameters associated with each method, namely, 
$\theta_1^*$, $\theta_2^*$ and $\theta_3^*$ for Algorithms~\ref{alg_mfp1}, \ref{alg_mfp2} and 
\ref{alg_qtz}, respectively, $\theta_3^*$ for the algorithm given by \eqref{qtznew} and 
$\theta_1^*$ for the algorithm given by \eqref{qtzmod}. These parameters are defined in 
Theorems \ref{th_mfp1}, \ref{th_mfp2} and \ref{conv_qtz} and the computational cost for computing them is 
the same as the cost for computing $\sigma_1$, the largest singular value of $A$.}

The left picture of Figure \ref{fig_fxpqtz} compares Algorithms~\ref{alg_mfp1}, \ref{alg_mfp2} and 
\ref{alg_qtz}, while the right one shows the performance of Algorithm~\ref{alg_mfp1} and the 
three variants of Quartz. We can see the equivalence between Quartz and New 
Quartz and also the equivalence between Modified Quartz and Algorithm~\ref{alg_mfp1}. 
\redcolor{Note that besides the advantage of QTZ* in terms of number of iterations, it does not need 
more arithmetic operations per iteration as we have seen in Table \ref{table2}.}
\begin{figure}[htbp]
\centering
\includegraphics[scale=0.315]{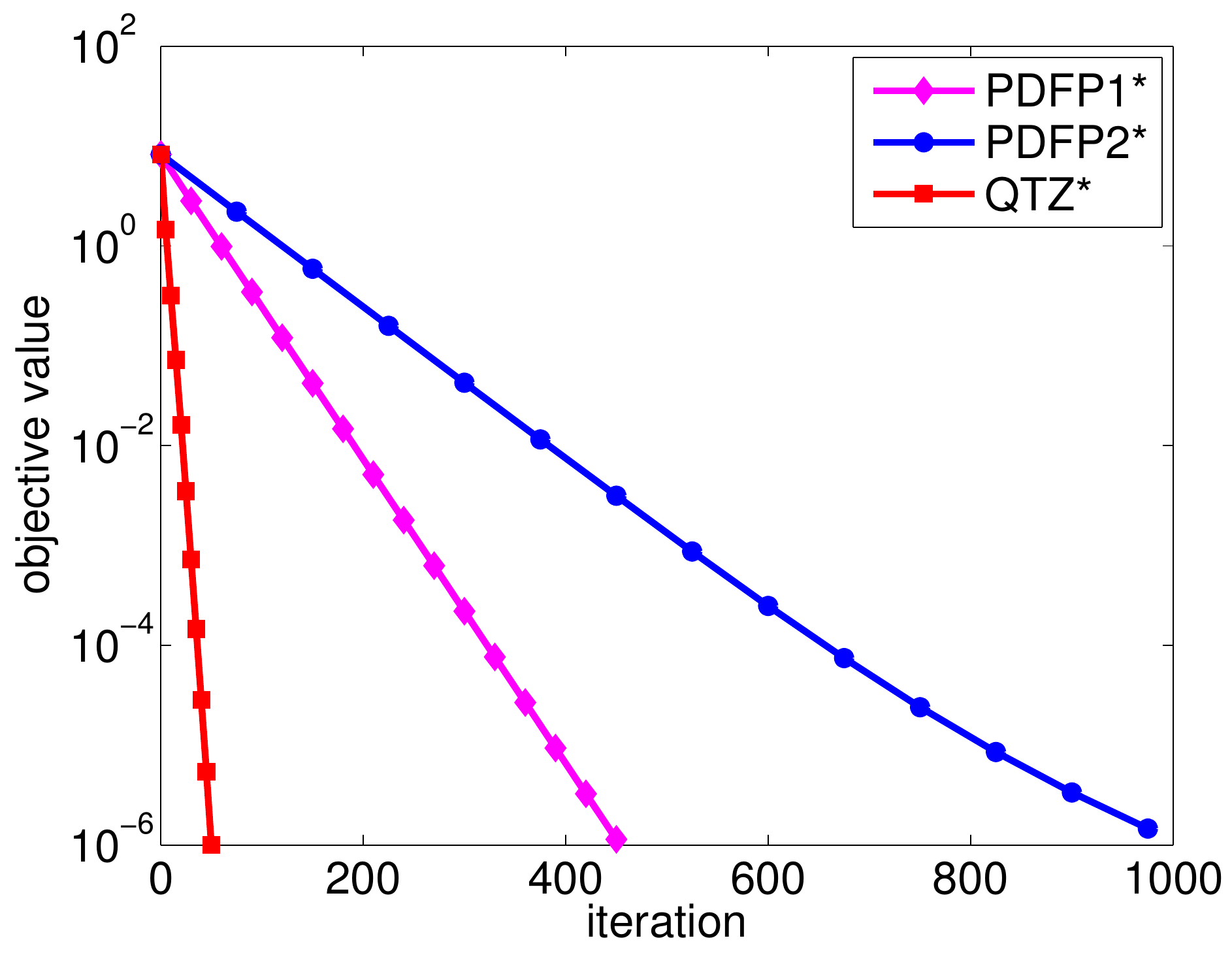}
\includegraphics[scale=0.315]{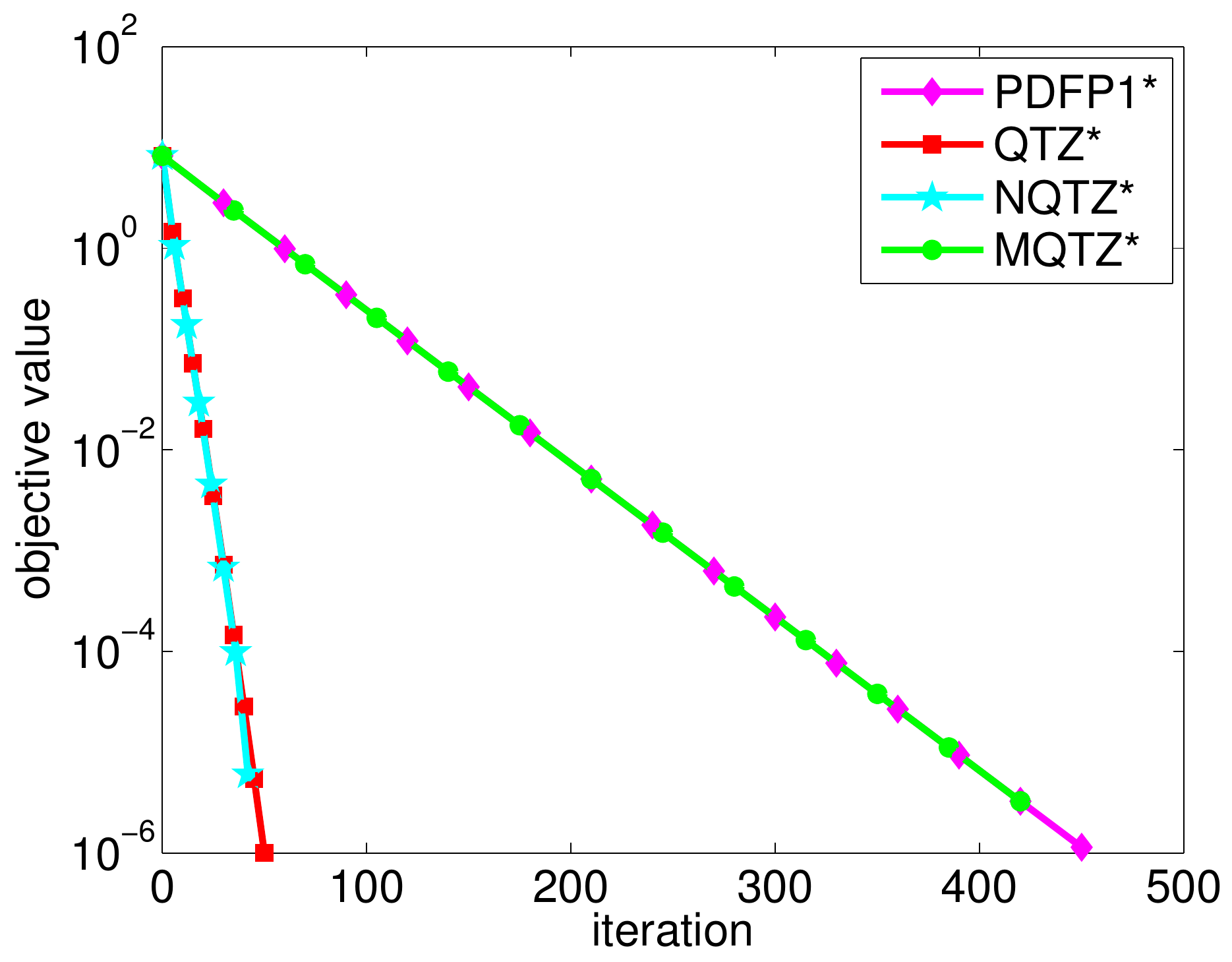}
\caption{\redcolor{Performance of the optimal versions of the algorithms
proposed in this paper applied to solve the problem \eqref{pd_prob}. 
The pictures show the objective values against the number of iterations. The dimensions considered were 
$d=10$, $m=1$ and $n=500$. The matrix $A\in\mathbb{R}^{d\times N}$ and the vector $y\in\mathbb{R}^{N}$ 
were randomly generated. For simplicity of notation we have denoted Algorithm~\ref{alg_mfp1} by 
PDFP1*, Algorithm \ref{alg_mfp2} by PDFP2*, Algorithm~\ref{alg_qtz} by QTZ* and the extensions \eqref{qtznew} 
(New Quartz) and \eqref{qtzmod} (Modified Quartz) by NQTZ* and MQTZ*, respectively.}}
\label{fig_fxpqtz}
\end{figure}

\redcolor{Despite the main goal of this work being  a theoretical study about convergence 
and complexity of various fixed point type methods, for the sake of completeness, we present here 
a comparison of our methods with the classical one for solving quadratic optimization problems: 
the conjugate gradient algorithm (CG).
Figure \ref{fig_fxpqtzcg} shows the performance of the optimal versions of the algorithms proposed 
in this paper compared with CG, applied to solve the problem \eqref{pd_prob}. On the top we have 
plotted the objective values against the number of iterations, while the bottom pictures 
show the objective values against the cpu time. The numerical experiments indicate that Quartz is 
competitive with CG. While Quartz needs more iterations than CG to converge, it is faster in runtime. This is due to the big difference between 
the effort per iteration of these two algorithms: 
$6dN+5d+9N$ arithmetic operations per iteration for Quartz compared to $4d^2+4N^2+4dN+14d+17N$ for CG.}

\begin{figure}[htbp]
\centering
\includegraphics[scale=0.312]{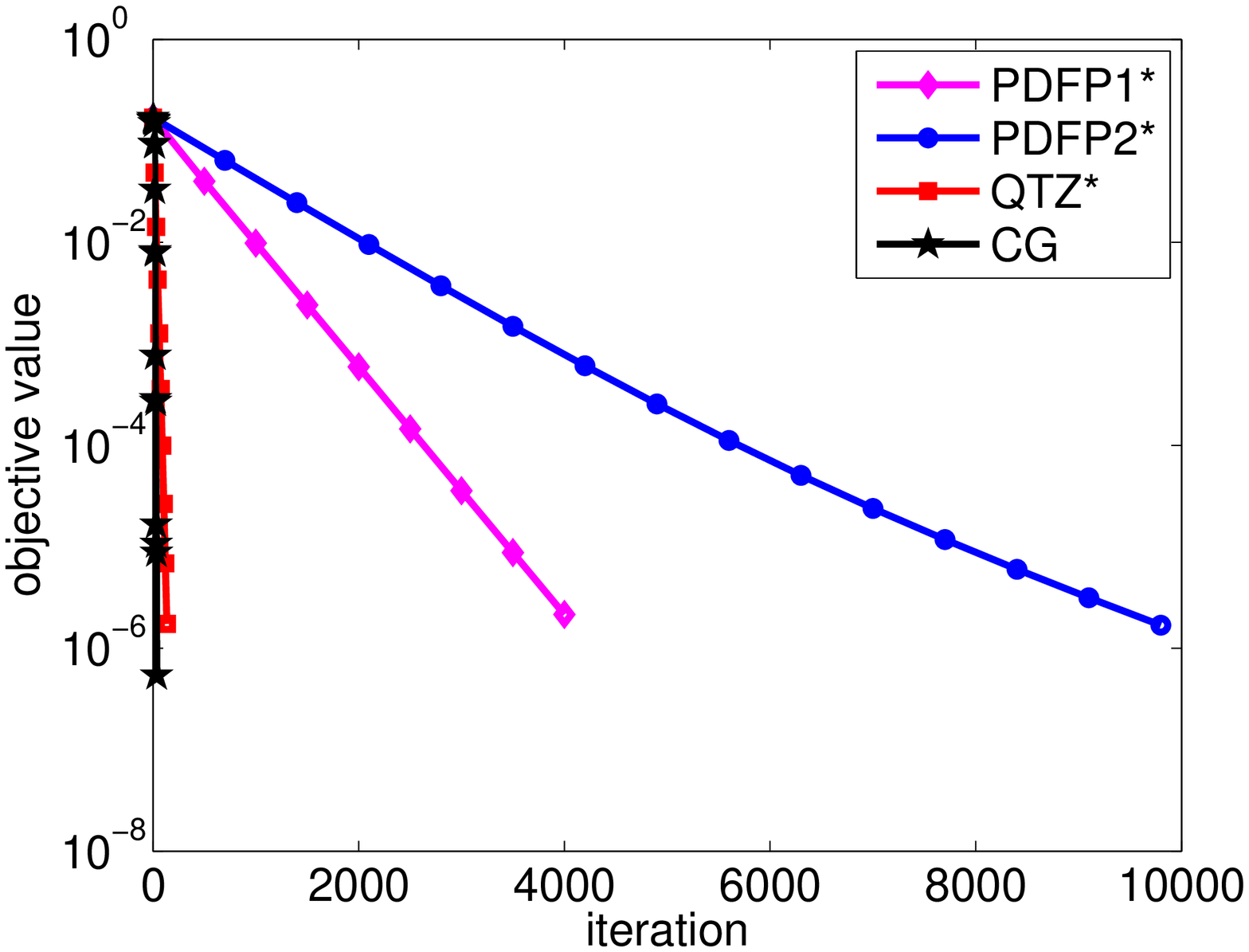}
\includegraphics[scale=0.312]{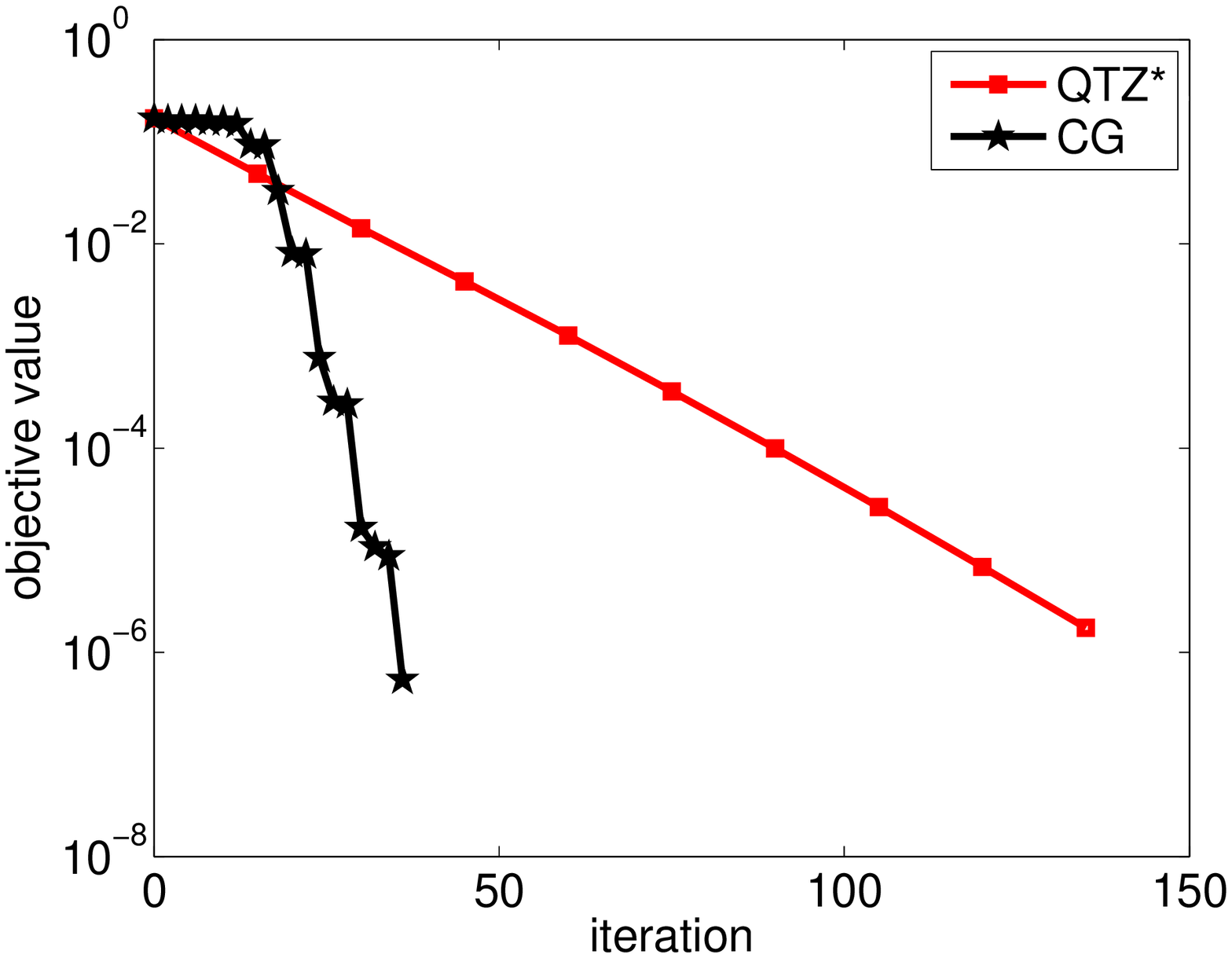}
\includegraphics[scale=0.312]{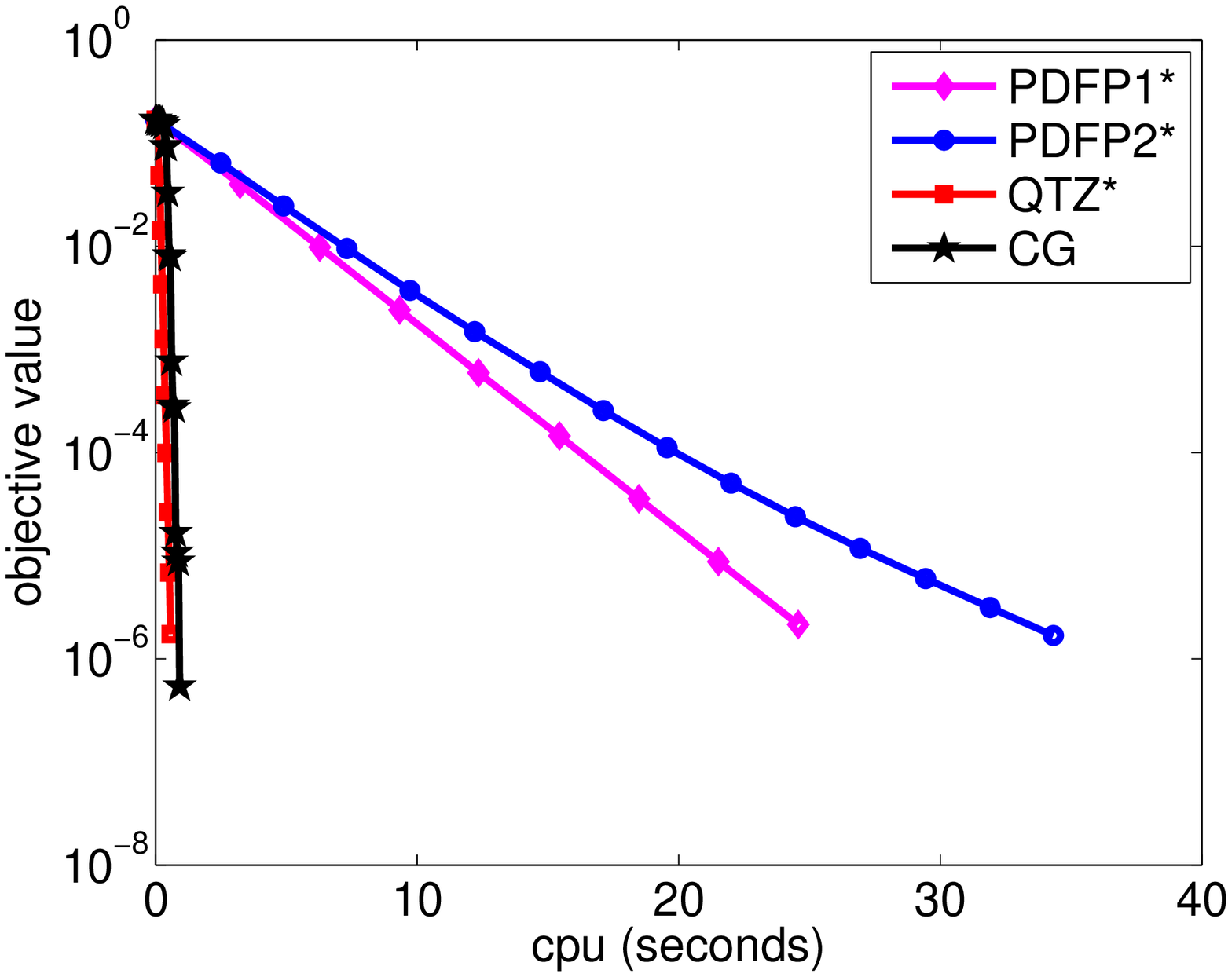}
\includegraphics[scale=0.312]{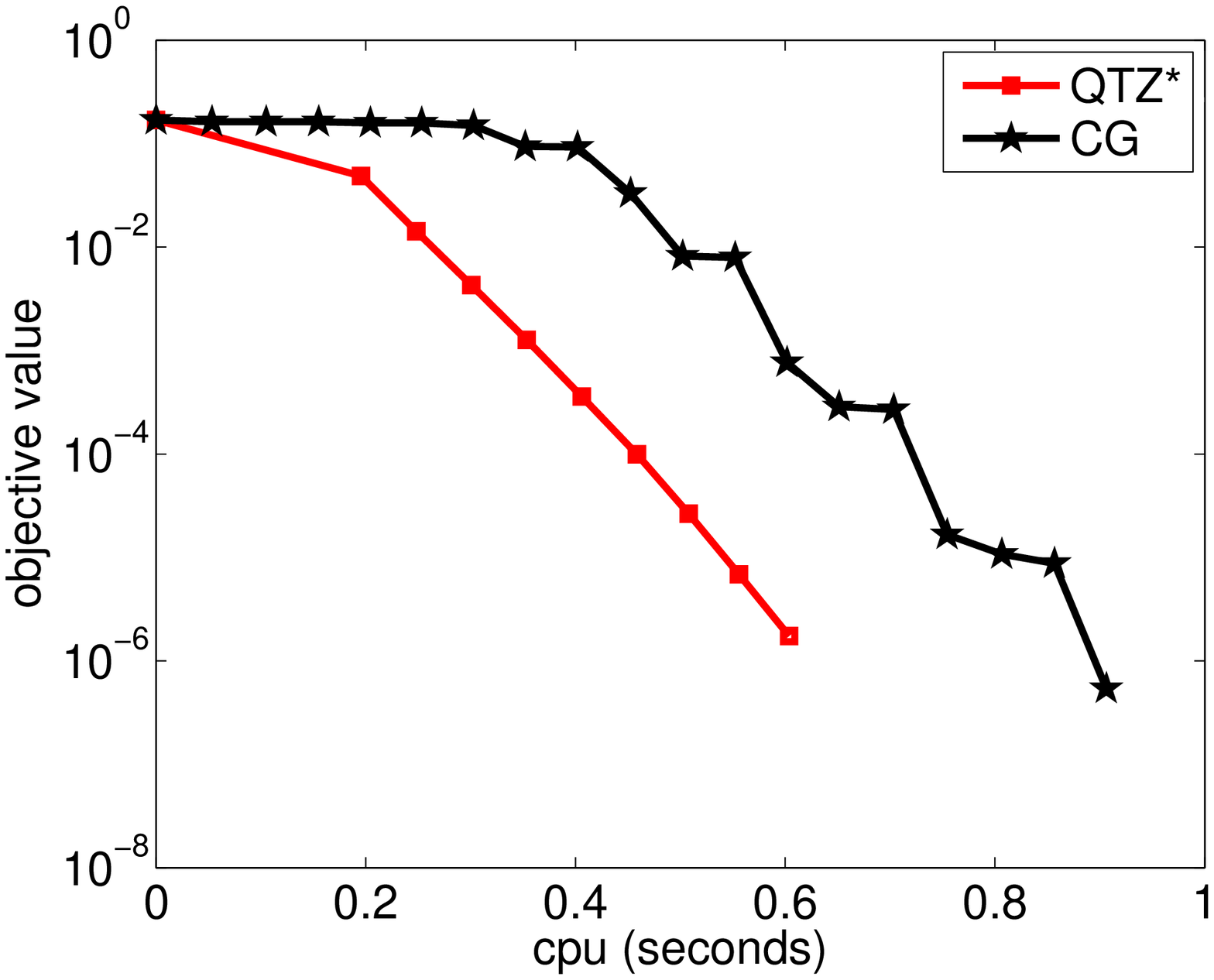}
\caption{Performance of the optimal versions of the algorithms proposed in this paper compared with 
the conjugate gradient algorithm, applied to solve the problem \eqref{pd_prob}. 
The dimensions considered were $d=200$, $m=1$ and $n=5000$. The matrix $A\in\mathbb{R}^{d\times N}$ 
and the vector $y\in\mathbb{R}^{N}$ were randomly generated. For simplicity of notation we have denoted 
Algorithm~\ref{alg_mfp1} by PDFP1*, Algorithm~\ref{alg_mfp2} by PDFP2*, Algorithm~\ref{alg_qtz} by QTZ* 
and conjugate gradient by CG. The pictures on the top show the objective values against the number 
of iterations, while the bottom ones show the objective values against the cpu time. 
The right pictures present the results of QTZ* and CG of the left ones with the horizontal axis rescaled. 
Note that despite QTZ* spent more iterations than CG, the computational time for solving 
the problem was less than that for CG.}
\label{fig_fxpqtzcg}
\end{figure}

\section{Conclusion}
\label{sec_concl} 
In this paper we have proposed and analyzed several algorithms for solving the ridge regression problem and its dual.  We have developed a (parameterized)  family  of 
fixed point methods applied to various equivalent reformulations of the optimality conditions. We have performed a convergence 
analysis and obtained complexity results for these methods, revealing  interesting geometrical insights between convergence speed and spectral properties of iteration matrices. Our main method achieves the optimal 
accelerated  rate of Nesterov. We have performed some numerical experiments to illustrate the properties of our 
algorithms \redcolor{as well as a comparison with the conjugate gradient algorithm. The numerical experiments 
indicate that our main algorithm is competitive with the conjugate gradient algorithm.}

%% The Appendices part is started with the command \appendix;
%% appendix sections are then done as normal sections
%% \appendix

%% \section{}
%% \label{}

%% If you have bibdatabase file and want bibtex to generate the
%% bibitems, please use

% \bibliographystyle{elsarticle-num} 
% \bibliography{references}

%% else use the following coding to input the bibitems directly in the
%% TeX file.

% \begin{thebibliography}{00}
% %% \bibitem{label}
% %% Text of bibliographic item
% \bibitem{}
% \end{thebibliography}

\end{document}